\documentclass[12pt,reqno]{amsart}

\usepackage{etex}
\usepackage[utf8]{inputenc}
\usepackage{amsfonts}
\usepackage{amsmath}
\usepackage{amssymb}
\usepackage{amsthm}
\usepackage{mathrsfs}
\usepackage{stmaryrd}
\usepackage{color}
\usepackage[english]{babel}
\usepackage[T1]{fontenc}
\usepackage{url}
\usepackage{graphicx}
\usepackage[pdftex,colorlinks,backref=page,citecolor=blue]{hyperref}
\usepackage{caption}
\usepackage{enumerate}
\usepackage{epstopdf}
\usepackage{hyphenat}
\usepackage{float}
\usepackage{indentfirst}
\usepackage[export]{adjustbox}
\usepackage{tikz}
\usepackage{color}
\usepackage[all]{xy}
\usetikzlibrary{matrix,arrows,decorations.pathmorphing}
\usepackage{booktabs}
\usepackage{array}
\usepackage{bm}

\newcolumntype{P}[1]{>{\centering\arraybackslash}p{#1}}

\allowdisplaybreaks

\newtheorem{theorem}{Theorem}[section]
\newtheorem{theoremletter}{Theorem}

\newtheorem{lemma}[theorem]{Lemma}
\newtheorem{proposition}[theorem]{Proposition}
\newtheorem{corollary}[theorem]{Corollary}

\theoremstyle{definition}
\newtheorem{assumption}[theorem]{Assumption}
\newtheorem{remark}[theorem]{Remark}
\newtheorem{notation}[theorem]{Notation}
\newtheorem{example}[theorem]{Example}
\newtheorem{definition}[theorem]{Definition}
\newtheorem{question}[theorem]{Question}

\newcommand{\RR}{\mathbb{R}}
\newcommand{\PP}{\mathbb{P}}

\newcommand{\bp}{\mathbf{p}}
\newcommand{\bq}{\mathbf{q}}
\newcommand{\bM}{\mathbf{M}}

\newcommand{\cO}{\mathcal{O}}
\newcommand{\cU}{\mathcal{U}}

\newcommand{\rH}{\mathrm{H}}

\DeclareMathOperator{\codim}{codim}
\DeclareMathOperator{\sing}{sing}
\DeclareMathOperator{\sign}{sgn}
\DeclareMathOperator{\reg}{reg}
\newcommand{\git}{/\! /}

\makeatletter
\def\l@subsection{\@tocline{2}{0pt}{2.5pc}{5pc}{}}
\makeatother

\author{Alessio Caminata}
\address{Dipartimento di Matematica, Universit\`a di Genova\\ via Dodecaneso 35, 16146, Genova, Italy}
\email{caminata@dima.unige.it}

\author{Han-Bom Moon}
\address{Department of Mathematics, Fordham University, New York, NY 10023}
\email{hmoon8@fordham.edu}

\author{Luca Schaffler}
\address{Dipartimento di Matematica e Fisica, Universit\`a degli Studi Roma Tre, Largo San Leonardo Murialdo 1, 00146, Roma, Italy}
\email{luca.schaffler@uniroma3.it}

\title[Determinantal varieties from point configurations on hypersurfaces]{Determinantal varieties from point configurations\\on hypersurfaces}

\subjclass{14M12, 14N20, 14J70, 14J17, 13C40}
\keywords{Determinantal variety, parameter space, point configuration, hypersurface, singularity.}

\begin{document}

\maketitle

\begin{abstract}
We consider the scheme $X_{r,d,n}$ parametrizing $n$ ordered points in projective space $\mathbb{P}^r$ that lie on a common hypersurface of degree $d$. We show that this scheme has a determinantal structure and we prove that it is irreducible, Cohen--Macaulay, and normal. Moreover, we give an algebraic and geometric description of the singular locus of $X_{r,d,n}$ in terms of Castelnuovo--Mumford regularity and $d$-normality. This yields a characterization of the singular locus of $X_{2,d,n}$ and $X_{3,2,n}$.
\end{abstract}

\section{introduction}


The central object of study in the current paper is the geometry of the parameter space of $n$-point configurations that lie on a degree $d$ hypersurface in $\mathbb{P}^r$, that is
\[
	X_{r, d, n} := \{\bp := (p_1,\ldots,p_n)\in(\PP^r)^n\mid p_1,\ldots,p_n~\textrm{lie on a degree $d$ hypersurface in}~\PP^r\}.
\]
We call this space the \emph{$n$-point configuration space for variable hypersurface}. It appears naturally in different contexts, ranging from the study of the geometry of data sets to questions in commutative algebra and algebraic geometry concerning determinantal varieties and invariant theory. Let us outline these connections more precisely.

\subsection{Motivation from the geometry of data sets}

A fundamental problem in data science is to understand whether a given data set has some geometric structure. Here a data set is a finite (usually massive and sparse) $n$-tuple $\bp := (p_1,\ldots, p_n)$ of points in an affine space $\mathbb{A}^{r}$ or a projective space $\PP^r$ with large  $r$. Being a large set of points in a higher dimensional space (the so-called \emph{point cloud}), analyzing the geometric structure of $\mathbf{p}$ is a highly nontrivial problem. Moreover, the inevitable existence of noise on the data makes finding such a hidden structure more difficult. To resolve this problem, many tools for detecting hidden structures or visualizing the data in a lower dimensional space have been sought. For instance, from the topological point of view, the persistent homology based on the filtered \v Cech complex, and Mapper based on the Reeb graph are well-developed tools for this purpose (see \cite{Car09, CVJ22}). In \cite{BKSW18}, the authors aim at understanding different properties of a real algebraic variety, such as dimension, degree, and irreducibility, given a finite sample of points on it. \cite{Gaf20} analyzes the computational complexity of finding the variety, within a certain class, that best fits a given set of points sampled with Gaussian noise from an unknown variety.

The work in this paper can be viewed as a study of the computational and algebro-geometric aspects of point clouds. For a given data set of $n$ points in projective space $\PP^{r}$, one of the first questions that one may ask is whether it satisfies a degree $d$ algebraic equation. It turns out this question is well-known and classical, and we revisit it in Proposition~\ref{prop:equationsdefiningXrdn}, where we provide an explicit set of sufficient and necessary conditions for a point cloud to lie on a common degree $d$ hypersurface. This can be thought of as the `noiseless case.' On the other hand, in the presence of noise in the data, an interesting question for a given point cloud $\bp$ is the following:

\begin{question}
\label{que:fundamentalquestions}

\

\begin{enumerate}
\item Do there exist $d$ and $\bp' \in (\PP^{r})^{n}$ such that $\bp'$ is sufficiently close to $\bp$ and $\bp' \in X_{r, d, n}$? In other words, up to some noise reduction, can we find $d$ such that $\bp$ satisfies a degree $d$ algebraic equation $F$?
\item Can we find the minimum $d$ and the interpolating algebraic equation $F$ effectively?
\end{enumerate}
\end{question}

This reminds us of Whitney's extension theorem in classical analysis: for a given point set $\bp \subseteq \RR^{r}$ and a function $f\colon\bp \to \RR$, find a function $F\colon\RR^{r} \to \RR$ such that $F|_{\bp} = f$, with a sufficient regularity and bounded norm (see \cite{Fef09} for an excellent survey). In some sense, what we seek in this project is an algebraic counterpart of Whitney's extension theorem (with $f = 0$). We expect that to address Question~\ref{que:fundamentalquestions} it is required to have an explicit and well-behaved set of defining equations of $X_{r, d, n}$ to both estimate the likelihood of the existence of $\bp' \in X_{r, d, n}$ and to effectively compute $\bp'$ and $F$. In fact, the equations of $X_{r, d, n}$ that we find in  Proposition~\ref{prop:equationsdefiningXrdn}, provide $X_{r, d, n}$ with a (non-generic) determinantal variety structure. We take advantage of this to investigate the algebraic and geometric properties of  $X_{r, d, n}$, and we hope that they could be used in future work also to investigate the questions mentioned above.


\subsection{Determinantal variety structure}

Determinantal varieties have been studied intensively and a rich literature is available, for example \cite{BV88,MR08} or the upcoming \cite{BCRV22}. Many nice properties of these varieties are known when the associated matrix consists of algebraically independent variables. In this case, the corresponding variety of $t$-minors is irreducible, Cohen--Macaulay, normal, has rational singularities, and the singular locus is cut out by the minors of size $t-1$. When the matrix does not have generic entries, the situation is more complicated, and the previously mentioned nice algebraic and geometric properties are known only for some special cases. For example, Hankel determinantal varieties, which arise as higher secant varieties of rational normal curves, are Cohen--Macaulay and normal \cite{Wat97}, with singular locus cut out by the minors of size one less, and they also have rational singularities \cite{CMSV18}.

The determinantal structure of $X_{r,d,n}$ can be seen as follows. If we fix homogeneous coordinates in $\PP^{r}$, then $X_{r,d,n}$ is cut out by the vanishing of the maximal minors of the $\binom{r+d}{d}\times n$ matrix whose $i$-th column is the image of the $i$-th point $p_i$ under the degree $d$ Veronese embedding $\PP^r\hookrightarrow\PP^{\binom{r+d}{d}-1}$ (Proposition~\ref{prop:equationsdefiningXrdn}). In Definition~\ref{def:Mrdn}, we call this matrix \emph{multi-Veronese matrix} and we endow $X_{r,d,n}$ with the scheme structure induced by the vanishing of these maximal minors. Such a matrix also appears in \cite[Section~2.2]{Gaf20}. Thus, we can see that $X_{r,d,n}$ has the structure of a determinantal variety associated to a non-generic matrix, as the entries within the same column are algebraically dependent.

\begin{theoremletter}[Corollary~\ref{cor:irranddimofXrdn}, Proposition~\ref{prop:XrdnisCMandGorenstein}, and Theorem~\ref{thm:normality}]\label{thm:maintheoremintro}
 	Let $r, d, n$ be positive integers and $r \ge 2$. The variety $X_{r, d, n}$ is a geometrically irreducible Cohen--Macaulay normal variety with an explicit set of defining equations. If $n < \binom{r+d}{d}$, $X_{r, d, n} = (\PP^{r})^{n}$. If $n \ge \binom{r+d}{d}$, $X_{r, d, n}$ is of dimension $\binom{r + d}{d} + n(r-1) - 1$. 
 \end{theoremletter}
 
Giving a complete description of the singular locus of $X_{r,d,n}$ seems to be a difficult task.
As for the generic determinantal variety, the vanishing locus of comaximal minors is a source of singularities for $X_{r, d, n}$. Geometrically, this means that if there is more than one degree $d$ hypersurfaces containing $\bp$, then $\bp$ is a singular point of $X_{r, d, n}$ (Theorem~\ref{thm:multiplehypersurfaces}). 
However, we are able to identify other singular points. We first introduce some notations that will be used throughout the paper.

\begin{notation}
\label{notations-used-in-the-paper}

Let $\mathbf{p}=(p_1,\ldots,p_n)\in(\mathbb{P}^r)^n$.

\begin{itemize}

\item[(i)] We denote by $Z_\mathbf{p}$ the closed reduced subscheme of $\mathbb{P}^r$ supported on the set $\{p_1,\ldots,p_n\}$. (Typically, we will be interested in considering $Z_\mathbf{p}$ when the points $p_i$ are distinct.) We denote the homogeneous ideal corresponding to $Z_\mathbf{p}$ simply by $I_\mathbf{p}$.

\item[(ii)] Assume that $\bp=(p_1,\ldots,p_n)$ lies on a unique hypersurface $S$ of degree $d$. We let $\mathbf{q}=(p_{i_1},\ldots,p_{i_k})$ be the $k$-tuple consisting of the points $p_i$ that lie in the singular locus of $S$. In short, we will simply write $\mathbf{q}=\mathbf{p}\cap S^{\mathrm{sing}}$ (note the abuse of notation).

\end{itemize}

\end{notation}

We show that if $\bq$ as introduced above satisfies a certain dependency condition, then $\bp$ is a singular point of $X_{r,d,n}$. Here, independence is defined in terms of $d$-normality or, equivalently, $(d+1)$-regularity (these notions are briefly surveyed in Appendix~\ref{sec:d-normality-m-regularity-secant-lines} for the reader's convenience).

\begin{theoremletter}[Theorem~\ref{thm:charactesizationsofsingularandsmoothnpointconfigurations}]\label{thm:matinthmsingular}
Let $\bp\in X_{r, d, n}$ be a point configuration that lies on a unique hypersurface $S$ of degree $d$. Let $\bq := \bp \cap S^{\sing}$. If two points in $\mathbf{q}$ coincide, then $\mathbf{p}$ is a singular point of $X_{r,d,n}$. If all the points in $\mathbf{q}$ are distinct, then the following statements are equivalent:
\begin{enumerate}
	\item $\bp$ is a smooth point of $X_{r, d, n}$;
	\item $Z_{\bq}$ is $d$-normal;
	\item $I_{\bq}$ is $(d+1)$-regular;
	\item $\reg I_{\bq}\leq d+1$.
\end{enumerate}
\end{theoremletter}

Based on this, in Corollary~\ref{cor:smoothnessXrdn} we give some concrete sufficient conditions which imply smoothness of $\mathbf{p}\in X_{r,d,n}$. We also note Corollary~\ref{cor:singularcriterion2}: if $\bq$ has a $(d+2)$-secant line, then $\bp \in X_{r, d, n}$ is singular. Unfortunately, describing the necessary and sufficient condition in terms of elementary point configurations seems to be out of reach --- we expect this to be related to the combinatorics of realizable matroids associated to $\bq$, whose classification is nearly impossible for large $r$ and $n$. However, we are able to provide a complete geometric characterization of the singular points of $X_{r,d,n}$ for small $r$ and $d$. For instance we show the following.

\begin{theoremletter}[Theorem~\ref{thm:characterization-sing-points-of-X2dn}]
Suppose that $\mathrm{char}\; \Bbbk = 0$. Let $n\geq \binom{2+d}{d}$. A point $\bp \in X_{2, d, n}$ is a smooth point if and only if there exists a unique degree $d$ curve $C$ passing through the points in $\bp$ and the points in $\bp$ that lie in the singular locus of $C$ are distinct.
\end{theoremletter}

In Theorem~\ref{thm:sings-of-TY-hypersurface} we provide an analogous characterization of the singular locus of $X_{3,2,n}$ when $\mathrm{char}\;\Bbbk\neq2$.

We point out that one may define a generalization $X_{r, d^{m}, n}$ as the set of point configurations lying on the intersection of $m$ linearly independent degree $d$ hypersurfaces in $\PP^{r}$. So, $X_{r, d^{1}, n} = X_{r, d, n}$. It turns out that $X_{r, d^{m}, n}$ also has the structure of a non generic determinantal variety. However, for $m \ge 2$, its geometric properties are not as good as the $m=1$ case. For instance, $X_{r, d^{m}, n}$ with $m \ge 2$ may be reducible in general (Example~\ref{ex:reducibleexample}), and some of its components are non-normal (Example~\ref{ex:Vdn} and Remark~\ref{rmk:non-normal-irr-comp}).


\subsection{Invariant theory and $X_{r,d,n}$}

The results discussed so far show that $X_{r, d, n}$ provides a nontrivial and interesting class of examples of non-generic determinantal varieties worth investigating in its own right. Furthermore, they form a class of examples of algebraic varieties with explicit coordinate rings. So their structure can be investigated with many classical and modern tools including commutative algebra, linear algebra, representation theory, and symbolic computation. In this direction, we observe that there is a natural  $\mathrm{SL}_{r+1}$-action on $X_{r,d,n}$ induced by that on $(\PP^{r})^{n}$. Thus, by the First Fundamental Theorem of Invariant Theory, the $\mathrm{SL}_{r+1}$-invariant defining equations of $X_{r,d,n}$ (or equivalently, the defining equations of $X_{r,d,n}\git \mathrm{SL}_{r+1}$) can be written in the form of brackets, i.e., polynomial combination of maximal minors of a generic $r\times n$ matrix. Lifting these equations to the Grassmann--Cayley algebra is related to classical theorems and problems in projective geometry. For example, writing in the Grassmann--Cayley algebra the unique defining equation for $X_{2,2,6}$ gives \emph{Pascal's theorem} about $6$ points on a plane conic (see \cite{CS21} for the precise statement and its higher dimensional generalization) and writing in the Grassmann--Cayley algebra the unique defining equation for $X_{3,2,10}$ is an open problem known as the \emph{Turnbull--Young Problem} \cite{TY27}. The case of $X_{2,2,n}$ was already studied in \cite{CGMS18}, where it is viewed as the Zariski compactification of the parameter space of point configurations that lie on a degree $2$ rational normal curve. In \cite{CGMS21}, a surprising connection with tropical geometry and phylogenetics in computational biology is explored.

\subsection{Structure of the paper}

The paper is organized as follows. In Section~\ref{sec:param-space-point-config-int-hyper}, we show that $X_{r, d^{m}, n}$ has the non-generic determinantal variety structure associated to the multi-Veronese matrix. Section~\ref{sec:onehypersurface} and \ref{sec:singularity} focus on the $m = 1$ case and Theorems~\ref{thm:maintheoremintro} and \ref{thm:matinthmsingular} are proved. We also provide some partial criteria for the singularities of $X_{r, d, n}$. The last section (Section~\ref{sec:examples}) is devoted to concrete examples of the general theory. More precisely, we study the singular locus of $X_{2, d, n}$ and $X_{3, 2, n}$. In Appendix~\ref{sec:d-normality-m-regularity-secant-lines} we provide some background on $d$-normality and Castelnuovo--Mumford regularity. 

\subsection*{Conventions}

We work over an arbitrary field $\Bbbk$, not necessarily algebraically closed. Unless otherwise stated, a point on a scheme is always a $\Bbbk$-rational point. 

\subsection*{Acknowledgements}

We would like to thank Andrea Bruno, Aldo Conca, Sandra Di Rocco, Mihai Fulger, Noah Giansiracusa, Kangjin Han, Andreas Leopold Knutsen, Wanseok Lee, and Maria Evelina Rossi. We also thank the anonymous referees for the valuable comments and suggestions. The first author is supported by the Italian PRIN2020 grant 2020355B8Y ``Squarefree Gr\"obner degenerations, special varieties and related topics'',  by the Italian PRIN2022 grant P2022J4HRR ``Mathematical Primitives for Post Quantum Digital Signatures'', by the INdAM--GNSAGA grant ``New theoretical perspectives via Gr\"obner bases'', and by the European Union within the program NextGenerationEU. The third author was partially supported by the projects ``Programma per Giovani Ricercatori Rita Levi Montalcini'', PRIN2017SSNZAW ``Advances in Moduli Theory and Birational Classification'', PRIN2020KKWT53 ``Curves, Ricci flat Varieties and their Interactions'', and, while at KTH, by a KTH grant by the Verg foundation. The first and third authors are members of the INdAM group GNSAGA. Much of the work was done while the second author was visiting Stanford University. He gratefully appreciates their hospitality during his visit.


\section{The parameter space of point configurations on the intersection of hypersurfaces}
\label{sec:param-space-point-config-int-hyper}

In this section, we define the parameter space $X_{r, d^{m}, n}$ of point configurations on intersections of linearly independent hypersurfaces. 

\begin{definition}\label{def:Xrdmn}
Let $r,d,m,n$ be positive integers. Let $X_{r,d^m,n}\subseteq(\PP^{r})^{n}$ be the set of point configurations $\bp = (p_{1},\ldots, p_{n})$, where $p_{1}, \ldots, p_{n}$ lie on the intersection of $m$ linearly independent degree $d$ hypersurfaces in $\PP^{r}$. For notational simplicity, we set $X_{r, d, n} := X_{r, d^{1}, n}$ and we call it \emph{$n$-point configuration space for variable hypersurface}.
\end{definition}

Throughout the paper, we let
\begin{equation*}
b(r, d) := \dim \rH^{0}(\PP^{r}, \cO(d)) = \binom{r+d}{d}.
\end{equation*}
Clearly, if $m \ge b(r,d)$, then $X_{r, d^{m}, n} = \emptyset$, because if so, a point configuration $\mathbf{p}\in X_{r, d^m, n}$ must lie on all degree $d$ hypersurfaces, which is impossible. Thus, from now on, we assume that $m < b(r, d)$. If $n < b(r, d) - m + 1$, then, by standard linear algebra, any $n$-point configuration $\bp$ lies on the intersection of $m$ linearly independent degree $d$ hypersurfaces. So, $X_{r, d^m, n} = (\PP^{r})^{n}$ in this case.

\begin{remark}\label{rmk:decreasingchain}
We have a decreasing chain of closed subsets
\[
X_{r, d^{1}, n} \supseteq X_{r, d^{2}, n} \supseteq \cdots \supseteq X_{r, d^{b(r, d)}, n} = \emptyset.
\]
\end{remark}

\medskip

In what follows, we endow the set $X_{r, d^{m}, n}$ with a natural determinantal scheme structure. In particular, we will have a complete set of defining equations for $X_{r,d^m,n}$.

\begin{definition}
\label{def:Mrdn}
Let $v\colon\PP^r\hookrightarrow\PP^{b(r,d)-1}$ be the degree $d$ Veronese embedding. We fix a homogeneous coordinate system on $\PP^{r}$ and $\PP^{b(r, d) - 1}$. If we have $n$ points $\bp=(p_1,\ldots, p_n)$ in $\PP^r$, we can arrange their images under $v$ as columns of the following $b(r, d) \times n$ matrix:
\begin{equation}\label{eqn:Mrdn}
	\bM_{r,d,n}(\bp)=(v(p_1),\ldots,v(p_n)).
\end{equation}
We call $\mathbf{M}_{r,d,n}(\mathbf{p})$ the \emph{multi-Veronese matrix}.
\end{definition}

\begin{proposition}\label{prop:equationsdefiningXrdn}
Suppose $n\geq b(r,d)-m+1$. Then $X_{r,d^m,n}\subseteq(\PP^{r})^{n}$ is the Zariski closed subset defined by the vanishing of the minors of size $b(r,d)-m+1$ of the matrix $\bM_{r,d,n}(\mathbf{x})$.
\end{proposition}

\begin{proof}
Let $\bp\in X_{r,d^m,n}$, which means by definition that there exist $m$ linearly independent hypersurfaces in $\PP^r$ containing the points in $\PP^r$. Equivalently, the linear system given by
\begin{displaymath}
\bM_{r,d,n}(\bp)^t\cdot\left( \begin{array}{c}
\vdots\\
a_I\\
\vdots
\end{array} \right)=
\left(\begin{array}{c}
\sum a_I p_1^I\\
\vdots\\
\sum a_I p_n^I
\end{array}\right)
= 
\left( \begin{array}{c}
0\\
\vdots\\
0
\end{array} \right)
\end{displaymath}
has $m$ linearly independent solutions. This is true if and only if the matrix $\bM_{r,d,n}(\bp)^t$ has kernel of dimension at least $m$, which means that the rank of $\bM_{r,d,n}(\bp)$ is at most $b(r,d)-m$, or equivalently, the minors of size $b(r,d)-m+1$ of $\bM_{r,d,n}(\bp)$ vanish.
\end{proof}

From now on we consider $X_{r,d^m,n}$ as a scheme with the structure induced by the vanishing of the minors of size $b(r,d)-m+1$ of $\mathbf{M}_{r,d,n}(\mathbf{x})$. We will later prove in Corollary~\ref{cor:reduced} that for $m=1$, the scheme $X_{r,d,n}$ is geometrically reduced.

\begin{remark}
\label{rmk:Xrdmn-as-deg-locus}
We note that $X_{r,d^m,n}$ can be viewed as an example of a \emph{degeneracy locus} (see \cite[Chapter~14]{Ful98}). More in detail, let $E$ be the trivial bundle on $(\mathbb{P}^r)^n$ of rank $b(r,d)$ and define
\[
F=\mathcal{O}_{(\mathbb{P}^r)^n}(d,0,\ldots,0)\oplus\mathcal{O}_{(\mathbb{P}^r)^n}(0,d,\ldots,0)\oplus\ldots\oplus\mathcal{O}_{(\mathbb{P}^r)^n}(0,0,\ldots,d).
\]
Let $\mathbf{M}_{r,d,n}\colon E\rightarrow F$ be the morphism of vector bundles that restricts to the fiber $E_{\mathbf{p}}\subseteq E$ over $\mathbf{p}\in(\mathbb{P}^r)^n$ giving the linear map $\mathbf{M}_{r,d,n}(\mathbf{p})^t\colon E_{\mathbf{p}}\rightarrow F_{\mathbf{p}}$. With this interpretation, $X_{r,d^m,n}$ is the $k$-th degeneracy locus of $\mathbf{M}_{r,d,n}$, that is
\[
X_{r,d^m,n}=\{\mathbf{p}\in(\mathbb{P}^r)^n\mid\mathrm{rank}(\mathbf{M}_{r,d,n}(\mathbf{p}))\leq k\},
\]
where $k=b(r,d)-m$. While this is an important interpretation of $X_{r,d^m,n}$, in this paper we do not use it to study $X_{r,d^m,n}$ as we want to give a self-contained account that focuses on the explicit set of defining equations.
\end{remark}

\begin{example}\label{ex:projectiveline}
For $\PP^{1}$, it is possible to give a simple description of $X_{1, d, n}$. A hypersurface in $\PP^{1}$ of degree $d$ is the union of at most $d$ distinct points. Thus, $X_{1, d, n}$ parametrizes point configurations whose support is at most $d$ distinct points. Formally, let $\mathcal{P}=\{I_1,\ldots,I_d\}$ be a partition of $[n]$, so that $I_1 \sqcup \ldots\sqcup I_d=[n]$. We define the corresponding diagonal embedding $\Delta_{\mathcal{P}}\colon(\mathbb{P}^1)^d\rightarrow(\mathbb{P}^1)^n$ as the map $(p_1,\ldots,p_d)\mapsto(q_1,\ldots,q_n)$, where for $i\in[n]$, $q_i:=p_j$, there exists a unique index $j\in[d]$ such that $i\in I_j$. Therefore, 
\[
	X_{1, d, n} = \bigcup_{\mathcal{P}}\Delta_{\mathcal{P}}((\PP^{1})^{d}) \subseteq (\PP^{1})^{n}, 
\]
where the union is over all the possible partitions $\mathcal{P}=\{I_1,\ldots,I_d\}$ of $[n]$. In particular, if $n > d$, $X_{1, d, n}$ is reducible. 
\end{example}

\begin{example}\label{ex:hyperplane}
When $d = 1$, $X_{r, 1^{m}, n}$ is the multiprojectivization of the generic determinantal variety. Geometrically, $X_{r, 1^{m}, n}$ is the set of point configurations lying on a codimension $m$ linear subspace of $\PP^{r}$. Denote the rank $k$ tautological bundle of $\mathrm{Gr}(k, \Bbbk^{r+1})$ by $\mathcal{S}_{k}$ (we consider the Grassmannian of $k$-dimensional subspaces in $\Bbbk^{r+1}$). Let $k=r+1-m$ and consider the natural morphism 
\[
f\colon\PP(\mathcal{S}_k) \times_{\mathrm{Gr}(k, \Bbbk^{r+1})}\ldots\times_{\mathrm{Gr}(k, \Bbbk^{r+1})}\PP(\mathcal{S}_{k}) \rightarrow (\PP^{r})^{n},
\]
where the domain is the fiber product of $n$ copies of $\PP(\mathcal{S}_{k})$ and the morphism is given by the fiber product of the compositions $\PP(\mathcal{S}_{k}) \hookrightarrow \PP^{r} \times \mathrm{Gr}(k, \Bbbk^{r+1})\rightarrow\PP^{r}$. By this construction, we have that the image of $f$ equals $X_{r, 1^{m}, n}$. Furthermore, the exceptional locus of $f$, $\mathrm{Exc}(f)$, equals the locus of point configurations on a subspace of codimension strictly larger than $m$. The image through $f$ of $\mathrm{Exc}(f)$ equals $X_{r, 1^{m+1}, n}$ and if $n \ge b(r, 1) - m + 1 = r + 2 - m$, then one can check that the codimension of $\mathrm{Exc}(f)$ is at least two. Since the domain of $f$ is smooth, these considerations imply that $X_{r, 1^{m}, n}$ is singular precisely along $X_{r, 1^{m+1}, n}$. So, we recover the well-known characterization of the singular locus of the generic determinantal variety (\cite[Theorem~2.6]{BV88}). For $d>1$, the singular locus of $X_{r, d^{m}, n}$ may be more complicated even for the case $m=1$, as we will see in Section~\ref{sec:singularity}.
\end{example}

In Section~\ref{sec:onehypersurface}, we will see that $X_{r, d, n}$ has a reasonably good local/global structure. However, when $m \geq 2$, in general $X_{r, d^{m}, n}$ has a very complicated nature. For the below examples, we assume that $\Bbbk$ is an algebraically closed field of characteristic different from $2$.

\begin{example}\label{ex:reducibleexample}
Consider $X_{2, 2^{2}, n}$ with $n \geq 6 $. For a general pair of distinct conics, its intersection is a set of four points not lying on a line. Let $X_{0} \subseteq X_{2, 2^{2}, n}$ be the closure of the subset parametrizing them. On the other hand, if a given pair of linearly independent conics has a one-dimensional intersection, then the intersection has to be the union of a line  $\ell$ and a point $p$, including the case that $p \in \ell$. Let $X_{1} \subseteq X_{2, 2^{2}, n}$ be the closure of the point configurations lying on a one-dimensional intersection. Clearly we have $X_{0} \cup X_{1} = X_{2, 2^{2}, n}$.

As $X_{0}$ parametrizes ordered $n$-point configurations whose support is at most four points in $\PP^{2}$, $X_{0}$ is 8-dimensional. On the other hand, $X_{1}$ has dimension $n+3$. So, if $n \ge 6$, then $X_{1} \nsubseteq X_{0}$ by dimension reasons. Since a general configuration parametrized by $X_{0}$ does not lie on a line union a point if $n \ge 4$, $X_{0} \nsubseteq X_{1}$. Therefore, $X_{2, 2^{2}, n}$ is reducible.

We point out that $X_{0}$ and $X_{1}$ are reducible. The former holds because $X_{0}$ is the union of the $8$-dimensional diagonals in $(\mathbb{P}^2)^n$. Let us prove the reducibility of $X_1$ by providing a decomposition into irreducible components. For a nonempty subset $I \subseteq [n]$, there is a morphism
\[
f_{I}\colon X_{2, 1, n-|I|} \times X_{2, 1^{2}, |I|} \cong X_{2, 1, n-|I|} \times \PP^{2} \rightarrow X_{2, 2^{2}, n}.
\]
If we denote $X_{1, I} := \mathrm{im}(f_{I})$, then we have that $X_{1} = \bigcup_{I}X_{1, I}$. As the domain of each $f_{I}$ is irreducible, $X_{1, I}$ is also irreducible. If $n-|I|<4$, then $X_{1,I}\subseteq X_0$. However, if $n-|I|\geq4$, then a general point of $X_{1, I}$ is not contained in $X_{0} \cup \bigcup_{J \ne I}X_{1, J}$. So, $X_{1, I}$ is an irreducible component of $X_{2, 2^{2}, n}$ and we can write:
\[
X_{2,2^2,n}=X_{0} \cup \bigcup_{\substack{I\subseteq[n],\\n-|I|\geq4}}X_{1,I}.
\]
\end{example}

\begin{example}
For any $d \ge 2$, two degree $d$ hypersurfaces may have a degree $k < d$ hypersurface in their intersection. This implies that $X_{r, k, n} \subseteq X_{r, d^{2}, n}$ for any $k < d$. However, we can see that $X_{r,k,n}$ is not an irreducible component of $X_{r, d^{2}, n}$ by using some ideas from Example~\ref{ex:reducibleexample}. More precisely, $X_{r,k,n}$ is a specialization of the subvariety of $X_{r,d^2,n}$ given by $X_{r,k,n-|I|}\times X_{r,(d-k)^2,|I|}$ for some nonempty $I\subseteq[n]$. We interpret the latter as parametrizing configurations $(p_1,\ldots,p_n)$ such that the points $p_1,\ldots,p_{n-|I|}$ lie on a degree $k$ hypersurface $S$ and $p_{n-|I|+1},\ldots,p_n$ lie on the intersection of two linearly independent hypersurfaces $S_1,S_2$ of degree $d-k$. Then $p_1,\ldots,p_n$ lie on the intersection of $S_1\cup S$ and $S_2\cup S$, which are linearly independent hypersurfaces of degree $d$. Note that $(p_1,\ldots,p_n)$ specializes to a point in $X_{r,k,n}$ by letting $p_{n-|I|+1},\ldots,p_n$ lie on $S$.
\end{example}

\begin{example}\label{ex:Vdn}
Here we consider $X_{3,2^3,n}$ for $n \ge 8$ (note that $X_{3,2^3,7} = (\PP^3)^7$). As in the case of $X_{2,2^2,n}$, we can define $X_i \subseteq X_{3,2^3,n}$ as the closure of the subvariety parametrizing point configurations lying on an $i$-dimensional intersection of three linearly independent quadrics (a net of quadrics). Then $X_0$ parametrizes point configurations lying on a complete intersection of three independent quadrics. Such configurations are called \emph{Cayley octads}, see \cite[Section~7]{PSV11}. So, if $n=8$, we have that $\dim X_0 = 21$ because to form a Cayley octad one can take $7$ general points in $\mathbb{P}^3$ and then the eighth one is uniquely determined by them \cite[Proposition~7.1]{PSV11}. For $n > 8$, $X_0$ is a finite cover of the above $21$-dimensional variety, hence $\dim X_0 = 21$ for all $n$.

A generic one-dimensional intersection of three independent quadrics is either a rational normal curve (if the net of quadrics does not have any reducible quadric) or a planar conic union two extra points (if two of the three quadrics have a common plane). We call the closure of the set parametrizing point configurations on the first kind of intersection by $V_{3,n}$, and the closure of the set parametrizing the latter by $W$. The variety $V_{3,n}$ is $(n + 12)$-dimensional and irreducible \cite[Lemma~2.2]{CGMS18}.

On the other hand, $X_2$ parametrizes point configurations on the intersection three reducible quadrics that have a common plane. A general point in $X_2$ is a point configuration on a plane and an extra point in $\PP^3$. Since a general point configuration on a rational normal curve has at most three points on a plane, we can see that $V_{3,n}$ is not included in $W$ or $X_2$. By dimension reason, $V_{3,n}$ is not contained in $X_0$ for $n>9$. For $n=9$, $\dim X_0 = \dim V_{3,9} = 21$ and a general point of $V_{3,9}$ is not in $X_0$. Therefore, $V_{3,n}$ is indeed an irreducible component of $X_{3,2^3,n}$ for $n\geq9$.
\end{example}

\begin{remark}
\label{rmk:non-normal-irr-comp}
By \cite[Proposition~4.28]{CGMS18} the variety $V_{3,n}$ is not normal for $n\geq8$. Therefore, by Example~\ref{ex:Vdn}, the parameter space $X_{3,2^3,n}$ has a non-normal irreducible component for $n\geq9$. So, in general, the irreducible components of $X_{r,d^m,n}$ may be non-normal. On the other hand, in Theorem~\ref{thm:normality} we will show that $X_{r, d, n}$ is a normal variety.
\end{remark}

It is evident that $X_{r, d^{m}, n}$ admits a natural $\mathrm{SL}_{r+1}$-action induced by that on $(\PP^{r})^{n}$. Let $L := \cO(1, \ldots, 1)|_{X_{r, d^{m}, n}}$. Then we obtain the GIT quotient 
\[
	X_{r, d^{m}, n}\git_{L}\mathrm{SL}_{r+1}.
\]
There is also a natural $S_{n}$-action on $X_{r, d^{m}, n}$ permuting $n$ points. 

\begin{example}
For some special parameters, $X_{r, d^{m}, n}$ or its GIT quotient is a classically well-known variety. For instance, $X_{2, 2, 6}\git_{L}\mathrm{SL}_{3}$ is known as \emph{Igusa quartic} or \emph{Castelnuovo--Richimond quartic} \cite[Remark~4.2]{Moo15}. Before taking the GIT quotient, $X_{2, 2, 6}$ is a hypersurface in $(\PP^{2})^{6}$ which encodes Pascal's theorem in classical projective geometry via Grassmann--Cayley algebra \cite{CS21}. More generally, $X_{2, 2, n}$ has been studied intensively in \cite{CGMS18}. If $n\geq9$, as we saw in Example~\ref{ex:Vdn}, $X_{3, 2^{3}, n}$ has $V_{3, n}$ as one of its irreducible components. Its GIT quotient $V_{3, n}\git_L\mathrm{SL}_{4}$ is a birational contraction of the moduli space of $n$-pointed stable rational curves $\overline{\mathrm{M}}_{0, n}$ (see \cite{GJM13}).
\end{example}

\begin{assumption}
From now on, we assume that $r \ge 2$, $d \ge 2$, and $m=1$.
\end{assumption}


\section{The case of one hypersurface}\label{sec:onehypersurface}

The scheme $X_{r, d, n}$ has a reasonably good geometric structure. In this section, we investigate its local/global geometric properties. Let $|\cO_{\PP^{r}}(d)|$ be the projective space parametrizing degree $d$ hypersurfaces in $\PP^{r}$. Let $\cU \subseteq |\cO_{\PP^{r}}(d)| \times \PP^{r}$ be the universal hypersurface with the structure morphism $\pi\colon\cU \to |\cO_{\PP^{r}}(d)|$. It is defined by a single equation 
\begin{equation}\label{eqn:universalequation}
F:=\sum_I a_Ix^I=0,
\end{equation}
where $x^I$ are the degree $d$ monomials in the coordinates of $\PP^r$ and $[\ldots:a_{I}:\ldots]$ are the homogeneous coordinates of $|\cO_{\PP^{r}}(d)|$. By the Jacobian criterion applied to $\mathcal{U}_{\overline{\Bbbk}}$, where $\overline{\Bbbk}$ is the algebraic closure of $\Bbbk$, we know that $\cU$ is smooth and of codimension one in $|\cO_{\PP^{r}}(d)| \times \PP^{r}$. For any point $h \in |\cO_{\PP^{r}}(d)|$, we have that $\pi^{-1}(h)$ is the hypersurface defined by $h$ in $\PP^{r}$. 

\begin{definition}\label{def:Uk}
For a positive integer $n$, let $\cU^{n}$ be the fiber product of $\cU$ over $|\cO_{\PP^{r}}(d)|$ $n$-times:
\[
\cU \times_{|\cO_{\PP^{r}}(d)|} \cU \times_{|\cO_{\PP^{r}}(d)|}\times \cdots \times_{|\cO_{\PP^{r}}(d)|} \cU.
\]
Naturally, $\cU^{n}$ can be regarded as a subscheme of $|\cO_{\PP^{r}}(d)| \times (\PP^{r})^{n}$. Note that $\cU^{1} = \cU$ and we set $\cU^0=|\mathcal{O}_{\PP^r}(d)|$.
\end{definition}

The fiber product $\cU^{n}$ parametrizes pairs $(S, \bp = (p_{1}, \ldots, p_{n}))$, where $S$ is a degree $d$ hypersurface in $\PP^{r}$ and $p_{i} \in S$ for all $i$. For each subset $I \subseteq [n]$, there is a natural forgetful map $f^{I}\colon\cU^{n} \to \cU^{n - |I|}$ forgetting the $j$-th point for all $j \in I$. For notational simplicity, we set $f^{k} = f^{\{k\}}$ for each $k \in [n]$ and $f := f^{[n]}\colon\cU^{n} \to |\cO_{\PP^{r}}(d)|$. On the other hand, by forgetting the hypersurface, we obtain a morphism $g\colon\cU^{n} \to (\PP^{r})^{n}$ whose set-theoretic image is $X_{r, d, n}$ with the reduced scheme structure. In summary, we have an incidence diagram 
\begin{equation*}
\xymatrix{&\cU^{n} \ar_{f}[ld] \ar^{g}[rd]\\ |\cO_{\PP^{r}}(d)| && X_{r, d, n},}
\end{equation*}
where both $f$ and $g$ are surjective. 

\begin{lemma}\label{lemma:Unirreducible}
For every $n$, $\cU^{n}$ is geometrically irreducible.	
\end{lemma}

\begin{proof}
It is enough to show irreducibility of $\mathcal{U}_{\overline{\Bbbk}}^n$ (see \cite[Chapter~II, Exercise~3.15]{Har77} for the definition of geometrically irreducible). So, we may assume that the base field is algebraically closed. We argue by induction on $n\geq1$. For $n=1$, $\cU$ is irreducible since it is smooth and connected.

Suppose that $\cU^{n-1}$ is irreducible. The map $\pi\colon\cU \to |\cO_{\PP^{r}}(d)|$ is flat because the Hilbert polynomials of the fibers are fixed. Since flatness is preserved after a base extension, the morphism $f^n\colon\cU^n=\cU^{n-1} \times_{|\cO_{\PP^{r}}(d)|} \cU \rightarrow \cU^{n-1}$ is also flat.

Assume by contradiction that $\cU^n$ is reducible. So we can write $\cU^n=C_1\cup C_2$ where $C_1,C_2\subseteq\cU^n$ are nonempty closed subsets that are not contained in each other. Observe that each $C_i$ contains at least one irreducible component of $\cU^n$. Define $U_i=\cU^n\setminus C_i$. Note that $f^n(U_i)$ is dense in $\cU^{n-1}$ for $i=1,2$ because any irreducible component of $\cU^n$ dominates $\cU^{n-1}$ by \cite[\S\,4, Lemma~3.7]{Liu02} (this is the geometric translation of the algebraic fact that a flat module is torsion-free). Moreover, $f^n(U_i)$ is open by \cite[Chapter~III, Exercise~9.1]{Har77} because $f^n$ is flat.

We define also $V_{n-1}\subseteq\mathcal{U}^{n-1}$ to be the subset of points $(S;p_1,\ldots,p_{n-1})$ such that $S$ is irreducible. As the subset of $|\mathcal{O}_{\mathbb{P}^r}(d)|$ corresponding to irreducible hypersurfaces is open for $r\geq2$, we have that $V_{n-1}$ is open in $\mathcal{U}^{n-1}$. We observe that if $y=(S;p_1,\ldots,p_{n-1})\in V_{n-1}$, then the fiber $F_y = (f^{n})^{-1}(y)$ is irreducible because it is isomorphic to $S$ as they have the same defining equations.

As $\mathcal{U}^{n-1}$ is irreducible, the intersection $f^n(U_1)\cap f^n(U_2)\cap V_{n-1}$ is a dense open subset of $\mathcal{U}^{n-1}$. By taking $y\in f^n(U_1)\cap f^n(U_2)\cap V^{n-1}$, we have that the fiber $F_y$ is irreducible and $F_y\cap U_1,F_y\cap U_2$ are two disjoint nonempty open subsets of $F_y$, which is a contradiction.
\end{proof}

\begin{lemma}
\label{lem:Un-is-gen-red}
For every $n$, the scheme $\cU^{n}$ is generically reduced.
\end{lemma}

\begin{proof}
The scheme $\cU^n$ is the complete intersection in $|\cO_{\PP^r}(d)|\times (\PP^r)^n$ defined by the vanishing of the pullbacks of $F$ given in \eqref{eqn:universalequation} with respect to the $n$ projections $|\mathcal{O}_{\mathbb{P}^r}(d)|\times(\mathbb{P}^r)^n\rightarrow|\mathcal{O}_{\mathbb{P}^r}(d)|\times\mathbb{P}^r$. We can write these pullbacks as follows. Let $x_i :=[x_{i,0}:\ldots:x_{i,r}]$ be the homogeneous coordinates of the $i$-th copy of $\PP^r$. Then $\cU^n$ is defined by the vanishing of the equations
\[
	F(x_1)=0,\ldots,F(x_n)=0.
\]
Since $\cU^n$ is geometrically irreducible by Lemma~\ref{lemma:Unirreducible}, it is enough to prove generic reducedness of $\mathcal{U}^n$ by showing that it is smooth on a nonempty open subset.

The Jacobian of $\cU^n$ is an $n \times (b(r, d) + (r+1)n)$ matrix of the form $[L|R]$, where $L=(x_i^I)_{i,I}$. Here $i=1,\ldots,n$ is indexing the rows and $I$, which runs through all the monomials of degree $d$, is indexing the columns. The matrix $R$ is given by
\begin{displaymath}
R=\left[\begin{array}{ccccc}
\nabla_{x_1} F(x_1) & 0 & \ldots & 0 \\
0 & \nabla_{x_2} F(x_2) & \ldots & 0 \\
\vdots & \vdots & \ddots & \vdots \\
0 & 0 & \ldots & \nabla_{x_n} F(x_n)
\end{array}\right].
\end{displaymath}
Note that the matrix $L$ is obtained by taking the partial derivatives with respect to $a_I$. For a generic choice of $[\ldots:a_I:\ldots]\in|\cO_{\PP^r}(d)|$, the corresponding hypersurface is smooth and the matrix $R$ at $\mathbf{p}$ has maximum rank. So, by the Jacobian criterion, $\mathcal{U}^n$ is smooth at a general point. In particular, $\cU^n$ is generically reduced.
\end{proof}

We now show that the morphism $g\colon\mathcal{U}^n\rightarrow X_{r,d,n}$ is birational for $n\geq b(r,d)$, which will be useful in many occasions.

\begin{lemma}\label{lem:localiso}
Suppose that $n \ge b(r, d)$. Let $Y \subseteq X_{r, d, n}$ be an open subset parametrizing point configurations lying on a unique hypersurface of degree $d$. Then the restriction $g_Y:=g|_{g^{-1}(Y)}$ of $g\colon\cU^n\rightarrow X_{r,d,n}$ is an isomorphism. 
\end{lemma}

Set theoretically, the bijectivity of $g_{Y}$ is clear. But note that we have a particular determinantal scheme structure on $X_{r, d, n}$, which is not known to be reduced at this point. 

\begin{proof}
To show that $g_{Y}$ is an isomorphism, we will construct a local inverse to $g_Y$ at each point in $Y$. Since the statement is local, we may assume that $Y = \mathrm{Spec}\; R$. Let us fix an arbitrary (not necessarily $\Bbbk$-rational) closed point $\bp\in \mathrm{Spec}\; R$, corresponding to a maximal ideal $\mathbf{m} \subseteq R$. 

For any $\mathbf{x}\in Y$, let $\sum_{I} a_{I}(\mathbf{x})X_{0}^{I_{0}}X_{1}^{I_{1}}\cdots X_{r}^{I_{r}}\in \rH^{0}(\cO_{\PP^{r}}(d))$ be an equation of the unique hypersurface through $\mathbf{x}$. It is sufficient to show that the coefficients $a_I(\mathbf{x})$ are regular functions in a neighborhood of $\bp$, that is, regular functions in $R$ after a certain localization. 

Let us consider the row-vector $A(\mathbf{x}) := (\ldots,a_{I}(\mathbf{x}),\ldots)$. Then, the multi-Veronese matrix $\bM(\mathbf{x}) := \bM_{r, d, n}(\mathbf{x})$ in \eqref{eqn:Mrdn}, whose columns are indexed by $1 \le i \le n$ and whose rows are indexed by $I$, satisfies $A(\mathbf{x})\bM(\mathbf{x})= \mathbf{0}$. As a $k(\bp) = R/\mathbf{m}$-matrix, $\mathrm{rank}\; \bM(\bp) = b(r,d) - 1$. So there exists a subset $J \subseteq [n]$ with $|J| = b(r,d) - 1$ such that the submatrix $\bM_{J}(\bp)$ consisting of the columns indexed by $J$ has the same rank as $\bM(\bp)$. Then there is an open neighborhood $Z_{0}$ of $\bp \in Y$ such that $\bM_{J}(\mathbf{x})$ is of full rank on $Z_{0}$. There is a row indexed by $K$ such that the square matrix $\bM_{\widehat{K},J}(\bp)$ obtained from $\bM_{J}(\bp)$ by eliminating the $K$-th row is still of full-rank square matrix. Let $D_{\widehat{K}, J}(\mathbf{x})$ be the determinant of $\bM_{\widehat{K},J}(\mathbf{x})$. Since $D_{\widehat{K}, J}(\bp) \ne 0 \in k(\bp) = R/\mathbf{m}$, we have $D_{\widehat{K}, J}(\bp) \notin \mathbf{m}$, hence it is not in the nilradical of $R$. Therefore, by restricting to the principal open subset corresponding to $D_{\widehat{K}, J}(\mathbf{x})$, we obtain a nonempty open neighborhood of $\bp$ contained in $Z_{0}$ where $D_{\widehat{K}, J}(\mathbf{x})$ is a unit. So, up to restricting $Z_0$ to a smaller open subset, we may assume that the determinant of $\bM_{\widehat{K}, J}(\mathbf{x})$ is a unit in $R$.

By \cite[Chapter~VII, Proposition~3.7]{Hun80}, the matrix $\bM_{\widehat{K}, J}(\mathbf{x})$ is an invertible $R$-matrix. Then the relation $A(\mathbf{x})\bM(\mathbf{x}) = \mathbf{0}$ implies $A_{\widehat{K}}(\mathbf{x})\bM_{\widehat{K},J}(\mathbf{x}) = -a_{K}(\mathbf{x})\bM_{K,J}(\mathbf{x})$, where $A_{\widehat{K}}(\mathbf{x})$ is the row-vector obtained by eliminating $a_{K}(\mathbf{x})$ and $\bM_{K,J}(\mathbf{x})$ is the $K$-th row of $\bM_{J}(\mathbf{x})$. Since $\bM_{\widehat{K},J}(\mathbf{x})$ is invertible, we have 
\[
A_{\widehat{K}}(\mathbf{x}) = -a_{K}(\mathbf{x})\bM_{K,J}(\mathbf{x})\bM_{\widehat{K},J}^{-1}(\mathbf{x}).
\] 
Then, all the entries of $A_{\widehat{K}}(\mathbf{x})$, up to a common factor $a_K(\mathbf{x})$, are regular functions on $Z_0$. Note that the argument does not rely on the reduced scheme structure $Y_{\mathrm{red}
}$.
\end{proof} 

\begin{corollary}\label{cor:birational}
If $n \ge b(r, d)$, $g\colon\cU^{n} \to X_{r, d, n}$ is birational. 
\end{corollary}

\begin{corollary}\label{cor:irranddimofXrdn}
If $n \ge b(r, d)$, $X_{r,d,n}$ is geometrically irreducible of dimension $b(r, d) - 1 + n(r-1)$. In particular, $X_{r,d,n}$ is irreducible over $\Bbbk$.
\end{corollary}

\begin{proof}
By Lemma~\ref{lemma:Unirreducible} we have that $X_{r,d,n}=g(\cU^n)$ is geometrically irreducible. Corollary~\ref{cor:birational} implies that
\[
	\dim(X_{r,d,n})=\dim(\cU^n)=b(r,d)-1+n(r-1),
\]
where $\dim(|\cO_{\PP^r}(d)|)=b(r,d)-1$ and each fiber of $\cU^n\rightarrow|\cO_{\PP^r}(d)|$ has dimension $n(r-1)$.
\end{proof}

The determinantal nature of $X_{r, d, n}$ provides nice local properties, too.

\begin{proposition}\label{prop:XrdnisCMandGorenstein}
The scheme $X_{r,d,n}$ is Cohen--Macaulay. Moreover, $X_{r,d,n}$ is Gorenstein if and only if $n \le b(r,d)$.
\end{proposition}

\begin{proof}
If $n < b(r, d)$, $X_{r, d, n} = (\PP^{r})^{n}$, so the statement is clear. Suppose that $n \ge b(r, d)$. By Corollary~\ref{cor:irranddimofXrdn} the codimension of $X_{r,d,n}$ in $(\PP^r)^n$ is $n-b(r,d)+1$. Therefore, on affine charts, $X_{r,d,n}$ is determinantal in the sense of \cite[\S\,18.5]{Eis95}. Hence, $X_{r,d,n}$ is Cohen--Macaulay by \cite[Theorem~1]{HE71}. Alternatively, Cohen--Macaulayness can be obtained using \cite[Theorem~14.3~(c)]{Ful98} in view of Remark~\ref{rmk:Xrdmn-as-deg-locus}. The proof that $X_{r,d,n}$ is Gorenstein if and only if $n=b(r,d)$ is analogous to the proof of \cite[Proposition~3.9]{CGMS18}, which easily adapts to the case of $X_{r,d,n}$.
\end{proof}

\begin{corollary}\label{cor:reduced}
The scheme $X_{r, d, n} \subseteq (\PP^{r})^{n}$ is geometically reduced. In particular, it is reduced over $\Bbbk$.
\end{corollary}

\begin{proof}
Let us base change to $\overline{\Bbbk}$. When $n < b(r, d)$, $X_{r, d, n} = (\PP^{r})^{n}$. Suppose $n \ge b(r, d)$. Note that the generic reducedness immediately follows from Corollary~\ref{cor:birational} and the generic reducedness of $\cU^n$ in Lemma~\ref{lem:Un-is-gen-red}. Furthermore, by Proposition~\ref{prop:XrdnisCMandGorenstein}, $X_{r, d, n}$ is Cohen--Macaulay. Thus, we have that $X_{r,d,n}$ is $R_{0}$ and $S_{1}$, so it is reduced.
\end{proof}

\begin{remark}\label{rmk:closure}
Let $U_{r, d, n} \subseteq (\PP^{r})^{n}$ be the subset consisting of $n$-tuples of distinct points lying on a smooth hypersurface. Clearly, $U_{r, d, n} \subseteq X_{r, d, n}$ and this is an open subset. By the irreducibility of $X_{r, d, n}$ and its reducedness in Corollary~\ref{cor:reduced}, we can conclude that $X_{r, d, n}$ is the scheme theoretic closure of $U_{r, d, n}$.
\end{remark}

We close this section by computing the multidegree of $X_{r, d, n}$ as a cycle on $(\mathbb{P}^r)^n$, when the base field is $\mathbb{C}$. Recall that for a closed subset $V\subseteq(\PP^r)^n$ one can define the \emph{multidegree} as the collection of intersection numbers of the form $([V]\cdot[L_1\times\cdots\times L_n])$ in the cohomology ring $\mathrm{H}^*((\PP^r)^n,\mathbb{Z})$, where $L_i\subseteq\PP^r$ are linear subspaces such that $\sum_{i=1}^n\dim(L_i)=\codim(V)$. Recently, multidegrees have found applications in several areas. For instance, in \cite{Bri03,CDNG15,CCRLMZ20,CCRC22} the notion of multidegree (or mixed multiplicity) plays a fundamental role.

\begin{proposition}
Assume that $\Bbbk=\mathbb{C}$. The multidegrees of $X_{r,d,n}\subseteq(\PP^r)^n$ are all zero, except for the intersections with classes of products of linear spaces consisting of $n-b(r,d)+1$ lines and $b(r,d)-1$ points, whose value is $d^{n-b(r,d)+1}$.
\end{proposition}

\begin{proof}
Consider $[L_1\times\cdots\times L_n]$ with $\sum_{i=1}^n\dim(L_i)=\codim(X_{r,d,n})=n-b(r,d)+1$. Notice that at least $b(r,d)-1$ of the linear spaces $L_i$ must be points. Suppose we have at least $b(r,d)$ points among the linear spaces $L_i$. If we take these points in general position, then there is no degree $d$ hypersurface in $\PP^r$ passing through them. Hence, $[X_{r,d,n}]\cdot[L_1\times\cdots\times L_n]=0$. So assume we have exactly $b(r,d)-1$ points among the linear spaces $L_i$, which forces the remaining linear spaces to be lines. If we take these points in general position, then there exists a unique smooth degree $d$ hypersurface $S$ in $\PP^r$ passing through them. If we take also the lines in general position, each one intersects $S$ in $d$ distinct points, hence $[X_{r,d,n}]\cdot[L_1\times\cdots\times L_n]=d^{n-b(r,d)+1}$.
\end{proof}


\section{Singularities of $X_{r, d, n}$}\label{sec:singularity}

As before, we assume that $r\geq 2$ and $n \ge b(r, d) = \binom{r+d}{d}$. In this section, we investigate two major sources of the singularities of $X_{r, d, n}$. One is the existence of multiple hypersurfaces passing through the given point configuration $\bp$, and the other one is the singularity of the hypersurface containing $\bp$. As a consequence, we show that $X_{r, d, n}$ is a normal variety in Theorem~\ref{thm:normality}.


\subsection{Multiple hypersurfaces through $\bp$}

Let $R=\Bbbk[x_1,\dots,x_N]$ be a polynomial ring over a field $\Bbbk$, and let $M=(f_{i,j})$ be an $s\times t$ matrix whose entries are polynomials in $R$. 

\begin{definition}\label{def:ideal-sing-loc-det-var-txt}
Fix $2\leq u\leq \min\{s,t\}$ and let $I_u(M)=(g_1,\ldots,g_k)$ be the ideal of $R$ generated by the $u\times u$ minors $g_i$ of $M$. Denote by $V_{u} :=Z(I_u(M))\subseteq \mathbb{A}^N$ the corresponding affine variety. Let $c$ be the codimension of $V_{u}$ inside $\mathbb{A}^N$. Let 
\[
	\mathrm{Jac}=\left( \frac{\partial g_i}{\partial x_j} \right)_{1 \le i \le k, 1 \le j \le N}
\]
be the corresponding Jacobian matrix of size $k \times N$. We define 
\[
	J_u :=I_u(M)+I_c(\mathrm{Jac}),
\]
which by \cite[Corollary 16.20]{Eis95}, defines the singular locus $V_{u}^{\mathrm{sing}}$ of $V_{u}$.
\end{definition}

\begin{lemma}\label{lem:derivative-determinant}
Suppose further that $s = t$. For any $1 \le j \le N$, we have
\[
	\frac{\partial \det M}{\partial x_j}\in I_{s-1}(M).
\]
That is, the partial derivatives of $\det M$ belong to the ideal of comaximal minors of $M$.	
\end{lemma}

\begin{proof}
We take the partial derivative with respect to $x_j$ and obtain
\[
\begin{split}
	\frac{\partial\det M}{\partial x_j} &= \sum_{\sigma\in \mathcal{S}_s}\sign(\sigma)\frac{\partial}{\partial x_j}\left(\prod_{i=1}^sf_{i,\sigma(i)}\right)\\
	&= \sum_{\sigma\in \mathcal{S}_s}\sign(\sigma)\sum_{\ell=1}^sf_{1,\sigma(1)}\cdots\frac{\partial f_{\ell,\sigma(\ell)}}{\partial x_j}\cdots f_{s,\sigma(s)}\\
	&= \sum_{\ell=1}^s\sum_{\sigma\in \mathcal{S}_s}\sign(\sigma)f_{1,\sigma(1)}\cdots\frac{\partial f_{\ell,\sigma(\ell)}}{\partial x_j}\cdots f_{s,\sigma(s)}\\
	&=\begin{vmatrix}
		\frac{\partial f_{1,1}}{\partial x_j} & \cdots & \frac{\partial f_{1,s}}{\partial x_j} \\ 
		f_{2,1} & \cdots & f_{2,s} \\ 
		\vdots	&   &\vdots \\ 
		f_{s,1} & \cdots & f_{s,s}
	\end{vmatrix}+\begin{vmatrix}
	f_{1,1} & \cdots & f_{1,s} \\ 
	\frac{\partial f_{2,1}}{\partial x_j} & \cdots & \frac{\partial f_{2,s}}{\partial x_j} \\ 
	\vdots	&   &\vdots \\ 
	f_{s,1} & \cdots & f_{s,s}
\end{vmatrix}+\cdots+\begin{vmatrix}
f_{1,1} & \cdots & f_{1,s} \\ 
	f_{2,1} & \cdots & f_{2,s} \\ 
\vdots	&   &\vdots \\ 
\frac{\partial f_{s,1}}{\partial x_j} & \cdots & \frac{\partial f_{s,s}}{\partial x_j}
\end{vmatrix}.
\end{split}
\]
By the cofactor expansion with respect to the row of partial derivatives, each one of the previous determinants belongs to $I_{s-1}(M)$, thus $\displaystyle\frac{\partial \det M}{\partial x_j}\in I_{s-1}(M)$ as well.
\end{proof}

\begin{lemma}\label{lem:singularities-from-smaller-minors}
We have
\[
	J_u\subseteq I_{u-1}(M).
\]	
In particular, $Z(I_{u-1}(M))$ is contained in $Z(J_u)$, which is the singular locus of $V_{u}$.
\end{lemma}

\begin{proof}
Recall $J_u=I_u(M)+I_c(\mathrm{Jac})$ and the containment $I_u(M)\subseteq I_{u-1}(M)$ always holds. Thus, it is enough to check that $I_c(\mathrm{Jac})\subseteq I_{u-1}(M)$. Observe that each entry of $\mathrm{Jac}$ belongs to $I_{u-1}(M)$. Indeed, an entry of $\mathrm{Jac}$ is  $\frac{\partial g_i}{\partial x_j}$, with $g_i$ a minor of size $u$ of $M$. Therefore, by Lemma~\ref{lem:derivative-determinant}, applied to the $u \times u$ minor whose determinant is $g_{i}$, $\frac{\partial g_i}{\partial x_j}\in I_{u-1}(M)$.
\end{proof}

\begin{remark}\label{rmk:properinclusion}
In general, the inclusion $Z(I_{u-1}(M)) \subseteq Z(J_{u}) = V_{u}^{\mathrm{sing}}$ is proper. This is because if $\bp \in Z(I_{u-1}(M))$, then the Jacobian matrix $\mathrm{Jac}$ at $\mathbf{p}$ is zero, while to be in $Z(J_{u})$, we only expect the rank of $\mathrm{Jac}$ to be strictly less than $u$.
\end{remark}

\begin{theorem}\label{thm:multiplehypersurfaces}
The variety $X_{r, d, n}$ is singular along $X_{r, d^2, n}$.
\end{theorem}

\begin{proof}
For $m=1,2$, let $X_{r,d^m,n}^{\mathrm{aff}}\subseteq(\mathbb{A}^{r+1})^n$ be the affine cone of $X_{r,d^m,n}$. The cone $X_{r,d,n}^{\mathrm{aff}}\subseteq(\mathbb{A}^{r+1})^{n}$ is the Zariski closed subset defined by the vanishing of the minors of size $b(r,d)$ of the multi-Veronese matrix $\bM_{r,d,n}(\mathbf{x})$ in \eqref{eqn:Mrdn}, while $X_{r,d^2,n}^{\mathrm{aff}}$ is the Zariski closed subset defined by the vanishing of the minors of size $b(r,d)-1$ of the same matrix. Therefore, by Lemma~\ref{lem:singularities-from-smaller-minors} we have that the points in $X_{r,d^2,n}^{\mathrm{aff}}$ are singular for $X_{r,d,n}^{\mathrm{aff}}$. It follows that $X_{r,d,n}$ must be singular along $X_{r,d^2,n}$ as well.
\end{proof}

In principle, the singular locus of $X_{r, d, n}$ can be calculated by using the Jacobian description in Definition~\ref{def:ideal-sing-loc-det-var-txt}, as we have an explicit collection of equations defining it. However, it seems to be difficult to make a geometrically equivalent statement for it. In the next section, we formulate necessary and sufficient conditions for a point $\mathbf{p}\in X_{r,d,n}$ to be smooth.


\subsection{Unique hypersurface through $\bp$}\label{ssec:uniquehypersurface}
We want to analyze more in detail the behavior of $n$-point configurations $\bp\in X_{r,d,n}$ lying on a unique hypersurface. Let $\mathbf{q}:=\mathbf{p}\cap S^{\mathrm{sing}}$ and $Z_\mathbf{q}$ and $I_\mathbf{q}$ as in Notation~\ref{notations-used-in-the-paper}. The first step is reformulating the smoothness of $\mathbf{p}$ in terms of the well known notions of $d$-normality and Castelnuovo--Mumford regularity. The necessary background is reviewed in Appendix~\ref{sec:d-normality-m-regularity-secant-lines}.

\begin{theorem}\label{thm:charactesizationsofsingularandsmoothnpointconfigurations}
Let $\bp\in X_{r, d, n}$ be a point configuration that lies on a unique hypersurface $S$ of degree $d$. Let $\bq := \bp \cap S^{\sing}$. If two points in $\mathbf{q}$ coincide, then $\mathbf{p}$ is a singular point of $X_{r,d,n}$. If all the points in $\mathbf{q}$ are distinct, then the following statements are equivalent:
\begin{enumerate}
	\item $\bp$ is a smooth point of $X_{r, d, n}$;
	\item $Z_{\bq}$ is $d$-normal;
	\item $I_{\bq}$ is $(d+1)$-regular;
	\item $\reg I_{\bq}\leq d+1$.
\end{enumerate}
\end{theorem}

\begin{proof}
By Lemma~\ref{lem:localiso}, locally near $\bp$, $X_{r,d,n}$ is isomorphic to $\cU^n$, so it suffices to check whether $(S,\bp)$ is smooth or singular in $\cU^n$. For this purpose, we consider the Jacobian $[L|R]$ of $\cU^n$, which was computed in the proof of Lemma~\ref{lem:Un-is-gen-red}.

Let $\bp=(p_1,\dots,p_n)$ and we may assume that $Z_{\bq} = \{p_{1}, \ldots, p_{k}\} \subseteq S^{\sing}$. Then, we have $\nabla_{x_{i}}F(x_{i})|_{(S,\bp)}= 0$ for $1 \le i \le k$, so the first $k$ rows of $R$ are identically zero. Thus, the Jacobian of $\cU^{n}$ is of full-rank if and only if the first $k$ rows of $L$ are linearly independent.

Therefore, if two points among $\mathbf{q}$ are equal, it follows that $\mathbf{p}$ is singular for $X_{r,d,n}$. So let us assume that $\mathbf{q}$ consists of distinct points. Then, the linear independence of the first $k$ rows of $L$ is equivalent to $Z_{\bq}$ imposing linearly independent conditions on $\rH^{0}(\cO_{\PP^{r}}(d))$, which is equivalent to the $d$-normality of $Z_{\bq}$. The equivalence among \textit{(2)}, \textit{(3)}, and \textit{(4)} is clear by Lemma~\ref{lem:d-normality-implies-d+1-regular} and the definition of Castelnuovo--Mumford regularity.
\end{proof}

\begin{remark}
As one of the referees pointed out, it is possible to argue Theorem~\ref{thm:charactesizationsofsingularandsmoothnpointconfigurations} using tools similar to \cite[IV.1]{ACGH85}. The space $P:=(\mathbb{P}^r)^n$ is the analogue of the symmetric power of a curve and the matrix $M:=\mathbf{M}_{r,d,n}(\mathbf{p})$ is the analogue of the Brill--Noether matrix. The assumption that $\mathbf{p}$ lies on a unique hypersurface $F=0$ is the statement that $\ker(M)$ is spanned by $F$, in which case $\mathrm{coker}(M)$ has dimension $n-b(r,d)+1$. We can then identify the tangent space $T_{\mathbf{p}}X$ with the kernel of the natural map $T_{\mathbf{p}}P\rightarrow\mathrm{coker}(M)$ which sends, using affine coordinates, a tuple of derivations $(D_1,\ldots,D_n)$ to $((D_1F)(p_1),\ldots,(D_nF)(p_n))$. One can then prove Theorem~\ref{thm:charactesizationsofsingularandsmoothnpointconfigurations} using this perspective.
\end{remark}

By Remark~\ref{rmk:basicpropertiesofregularity}~(1), we obtain the following corollary.

\begin{corollary}\label{cor:singularcriterion}
Let $\bp\in X_{r, d, n}$ be a point configuration that lies on a unique hypersurface $S$ of degree $d$. Let $\bq := \bp \cap S^{\sing}$. If $|\bq| > b(r, d)$, then $\bp \in X_{r, d, n}^{\mathrm{sing}}$.
\end{corollary}

\begin{corollary}\label{cor:singularcriterion2}
Let $\bp\in X_{r, d, n}$ be a point configuration that lies on a unique hypersurface $S$ of degree $d$ and let $\bq := \bp \cap S^{\sing}$. If $\bq$ has $(d+2)$-secant line, then $\bp \in X_{r, d, n}^{\mathrm{sing}}$.
\end{corollary}

\begin{proof}
By Theorem~\ref{thm:charactesizationsofsingularandsmoothnpointconfigurations} we can assume that the points in $\mathbf{q}$ are distinct. If $\mathbf{q}$ has a $(d+2)$-secant line, then $\mathrm{reg}(I_{\mathbf{q}})\geq d+2$ by Proposition~\ref{prop:regularityandsecant}. Hence, $\mathbf{p}$ is singular by Theorem~\ref{thm:charactesizationsofsingularandsmoothnpointconfigurations}.
\end{proof}

For a finite scheme $\Gamma$ in $\PP^r$, let $\langle \Gamma\rangle \subset \PP^r$ be the smallest linear subspace containing the support of $\Gamma$. When $|\bq|$ is relatively small or $Z_{\bq}$ is sufficiently `linearly general' in $\langle Z_{\bq}\rangle$, we can phrase some sufficient conditions for smoothness of a point $\mathbf{p}\in X_{r,d,n}$. For a finite subscheme $\Gamma \subseteq \PP^{r}$, we say $\Gamma$ is \emph{$k$-general} if for any $\Gamma' \subseteq \Gamma$ with $|\Gamma'| = k+1$, $\dim \langle \Gamma' \rangle = k$. In other words, $\Gamma$ is $k$-general if any $(k+1)$-subscheme is linearly independent. Let $t(\Gamma)$ be the largest $k$ such that $\Gamma$ is $k$-general. For instance, a reduced finite scheme $\Gamma$ has $t(\Gamma) = 1$ if there are three points on a line (so there is a 3-secant line).  

\begin{corollary}\label{cor:smoothnessXrdn}
Let $\bp \in X_{r, d, n}$ be a point configuration lying on a unique hypersurface $S$ of degree $d$ and $\bq = \bp \cap S^{\sing}$. Assume that all the points in $\mathbf{q}$ are distinct. Then the following hold:
\begin{enumerate}[(i)]
\item If $|Z_{\bq}| - \dim \langle Z_{\bq} \rangle \le d$, then $X_{r, d, n}$ is smooth at $\bp$.
\item If $Z_{\bq}$ is a set of at most $d+1$ distinct points, then $X_{r, d, n}$ is smooth at $\bp$. 
\item Suppose $\mathrm{char} \;\Bbbk = 0$. If the following inequality holds:
\begin{equation}\label{eqn:normalitybound}
	d\geq\left\lceil \frac{|Z_{\bq}| - \dim \langle Z_{\bq} \rangle - 1}{t(Z_{\bq})}\right\rceil+ 1, 
\end{equation}
then $\bp$ is a smooth point of $X_{r, d, n}$.
\item Suppose $\mathrm{char} \;\Bbbk = 0$. Assume that $Z_{\bq}$ is in linearly general position in $\langle Z_{\bq} \rangle$ and $d \ge \left\lceil \frac{|Z_{\bq}| -1}{\dim \langle Z_{\bq} \rangle}\right\rceil$. Then $\bp$ is a smooth point of $X_{r, d, n}$.
\end{enumerate}
\end{corollary}

\begin{proof}
For part (i), $Z_{\bq}$ is $(|Z_{\bq}| - \dim \langle Z_{\bq} \rangle)$-normal by Remark~\ref{rmk:basicpropertiesofregularity}~(3). Since $(|Z_{\bq}| - \dim \langle Z_{\bq} \rangle)\leq d$, then $Z_{\bq}$ is $d$-normal by Corollary~\ref{cor:d-normal-then-d+1-normal}. Therefore, $\mathbf{p}$ is a smooth point of $X_{r,d,n}$ by Theorem~\ref{thm:charactesizationsofsingularandsmoothnpointconfigurations}. Part (ii) follows from part (i) above because if $Z_{\bq}$ is a set of at most $d+1$ distinct points, then $|Z_{\bq}| - \dim \langle Z_{\bq} \rangle\leq d+1-1=d$. Part (iii) follows from \cite[Proposition~2.1]{Kwa00} combined with Theorem~\ref{thm:charactesizationsofsingularandsmoothnpointconfigurations}. In \cite{Kwa00} it is assumed the algebraic closedness of $\Bbbk$, but we may lift this condition by Lemma~\ref{lem:basechange}. Finally, part (iv) is an immediate application of (iii) with $t(Z_{\bq}) = \dim \langle Z_{\bq} \rangle$.
\end{proof}


\subsection{Normality of $X_{r,d,n}$}\label{ssec:normal}

As the first application of Theorem~\ref{thm:charactesizationsofsingularandsmoothnpointconfigurations}, we establish the normality for $X_{r,d,n}$. This was known for $X_{2,2,n}$ by \cite[Theorem~3.8]{CGMS18}. We need a few preliminaries. Recall from Section~\ref{sec:onehypersurface} that we have an incidence diagram 
\[
	\xymatrix{&\cU^{n} \ar[ld]_{f} \ar[rd]^{g}\\
	|\cO_{\PP^{r}}(d)|&& X_{r,d,n}.}
\]

\begin{definition}\label{def:defofHp}
For each $\bp \in X_{r, d, n}$, let
\[
	H_{\bp} := f(g^{-1}(\bp)) = \PP\rH^{0}(I_{\bp}(d)).
\]
\end{definition}

\begin{lemma}\label{lem:reducedhypersurface}
Let $\bp=(p_1,\ldots,p_n)\in X_{r, d, n}$ and $S \in |\cO_{\PP^{r}}(d)|$ be a hypersurface that contains $\bp$. If $S$ has a non-reduced component, then we can find a reduced hypersurface $S' \in |\cO_{\PP^{r}}(d)|$ that contains $\bp$. Furthermore, $\dim H_{\bp} \ge r$. 
\end{lemma}

\begin{proof}
Suppose that $S = V(\prod_if_{i}^{m_{i}})$, where the $f_i$ are distinct irreducible polynomials. Since $S$ is non-reduced, one of the $m_{i}$ is strictly larger than one. Note that $\{p_1,\ldots,p_n\} \subseteq S_\mathrm{red} = V(\prod f_{i})$. Then for any general reduced hypersurface $S'$ with degree $e := \sum_i(m_{i}-1)\deg(f_i)$, $\{p_1,\ldots,p_n\}\subseteq S' \cup S_\mathrm{red}$, where $S' \cup S_\mathrm{red}$ has degree $d$ and is reduced. Furthermore, $\dim H_{\bp} \ge \dim |\cO_{\PP^{r}}(e)| \ge \dim |\cO_{\PP^{r}}(1)| = r$. 
\end{proof}

\begin{lemma}\label{Xrd^2n-codim-2-in-Xrdn}
Let $n \ge b(r, d)$. Then the codimension of $X_{r,d^2,n}$ in $X_{r,d,n}$ is at least $2$.
\end{lemma}

\begin{proof}
For $n\geq b(r,d)$, the morphism $g\colon\cU^n\rightarrow X_{r,d,n}$ is birational (Corollary \ref{cor:birational}). The fact that $g^{-1}(X_{r,d^2,n})$ is contained in the exceptional locus of $g$ implies that $\dim X_{r,d^2,n} <\dim g^{-1}(X_{r,d^2,n})$. Additionally, $g^{-1}(X_{r,d^2,n})$ is a proper closed subset of $\cU^n$, and the latter is irreducible by Lemma~\ref{lemma:Unirreducible}. So we have $\dim(g^{-1}(X_{r,d^2,n}))<\dim(\cU^n)$. The birationality of $g$ also implies that $\dim(\cU^n)=\dim(X_{r,d,n})$. Since
\[
\dim(X_{r,d^2,n})<\dim(g^{-1}(X_{r,d^2,n}))<\dim(\cU^n)=\dim(X_{r,d,n}),
\]
we can conclude that $X_{r,d^2,n}$ has at least codimension $2$ in $X_{r,d,n}$.
\end{proof}

\begin{theorem}\label{thm:normality}
The variety $X_{r,d,n}$ is normal for all $n$.
\end{theorem}

\begin{proof}
As $X_{r,d,n}=(\PP^r)^n$ for $n<b(r,d)$, we only need to consider $n\geq b(r,d)$. We know from Proposition~\ref{prop:XrdnisCMandGorenstein} that $X_{r, d, n}$ is Cohen--Macaulay. So, by Serre's criterion, it is sufficient to show that $X_{r, d, n}$ is regular in codimension one.

Let $\bp\in X_{r,d,n}$ be a singular point. As we already know from Lemma~\ref{Xrd^2n-codim-2-in-Xrdn} that $X_{r,d^2,n}$ has codimension at least $2$ in $X_{r,d,n}$, we only need to consider the case where $\bp$ lies on a unique hypersurface $S$. Observe that $S$ does not have any non-reduced component by Lemma~\ref{lem:reducedhypersurface}. So, the singular locus of $S$ is at least of codimension one. By Theorem~\ref{thm:charactesizationsofsingularandsmoothnpointconfigurations}, in order for $\bp$ to be singular, we must have at least two points in $\bp$ in the singular locus of $S$, which is a condition of codimension at least two. Therefore the codimension of the potential singular loci are at least two and $X_{r, d, n}$ is regular in codimension one.
\end{proof}

\begin{remark}\label{rmk:singularityofUn}
It follows from the proof of Theorem~\ref{thm:charactesizationsofsingularandsmoothnpointconfigurations} that $\cU^{n}$ is regular in codimension one. This is because for a point $(S,\mathbf{p})\in\mathcal{U}^n$ to be singular we must have that $\mathbf{q}=\mathbf{p}\cap S^{\sing}$ consists of at least two points (either equal or not). Since $\cU^{n}$ is a complete intersection, it is a normal Gorenstein variety.
\end{remark}

\begin{proposition}
\label{rmk:Qfactorility}
If $n \geq b(r, d)$, then $X_{r, d, n}$ is not $\mathbb{Q}$-factorial. 
\end{proposition}

\begin{proof}
Consider the birational morphism $g\colon\mathcal{U}^n\rightarrow X_{r,d,n}$. Since $g$ forgets the hypersurface containing the point configuration, we know that $E := g^{-1}(X_{r,d^2,n})$ is the exceptional locus of $g$. As $X_{r,d,n}$ is normal by Theorem~\ref{thm:normality}, we can conclude that $X_{r,d,n}$ is not $\mathbb{Q}$-factorial by \cite[Section~1.40]{Deb01} if $E$ has a component of codimension at least two in $\mathcal{U}^n$. In what follows, we show the existence of such a component of $E$.

We first introduce some notations. Given $\mathbf{p}\in(\mathbb{P}^r)^n$, let $H_{\mathbf{p}}\subseteq|\mathcal{O}_{\mathbb{P}^r}(d)|$ as in Definition~\ref{def:defofHp}. By the definition of $E$, for every $(S, \mathbf{p})\in E$ we have that the projective dimension of $H_{\mathbf{p}}$ is at least one. Let $W_j:= \{(S, \mathbf{p})\mid\dim H_{\mathbf{p}} = j\}$. Clearly, $E=\cup_{j \geq1}W_j$ and the union is finite. Note that $W_j$ may be reducible and that $W_{j+1}$ is in the closure of $W_j$. So, it is enough to consider $W_1$.

If $\dim H_{\mathbf{p}} = 1$, then the intersection $\Sigma_{\mathbf{p}} = \cap_{S \in H_{\mathbf{p}}}S$ can be of codimension one or two in $\mathbb{P}^r$. Let $W_1^{(i)} \subseteq W_1$ be the subset such that $\mathrm{codim}_{\mathbb{P}^r}\Sigma_{\mathbf{p}}=i$ for $i\in\{1,2\}$. Then $W_{1}$ is the disjoint union of $W_{1}^{(1)}$ and $W_{1}^{(2)}$. We show that: (a) $W_1^{(2)}$ is open in $W_1$; (b) The codimension of $W_1^{(2)}$ in $\mathcal{U}^n$ is at least $2$. These imply that one of the irreducible components of the closure of $W_{1}^{(2)}$ is an irreducible component of $E$ of codimension at least two in $\mathcal{U}^n$.

Let us prove (a). Consider the universal family $\mathcal{V}\rightarrow\mathrm{Gr}(2,\rH^0(\mathcal{O}_{\mathbb{P}^r}(d)))=:\mathrm{Gr}$ parametrizing the intersection of two degree $d$ hypersurfaces. Denote by $\mathcal{V}^n$ the fiber product of $\mathcal{V}$ over $\mathrm{Gr}$ with itself $n$-times. We have that $\mathcal{V}^{n}$ parametrizes the pairs $(H_{\mathbf{p}}, \mathbf{p})$ consisting of a one-dimensional linear system of hypersurfaces $H_{\mathbf{p}}$ and a point configuration $\mathbf{p}$ lying on the intersection of two distinct hypersurfaces in $H_{\mathbf{p}}$. In particular, we have a map $W_1\rightarrow\mathcal{V}^n$ given by $(S, \mathbf{p}) \mapsto (H_{\mathbf{p}}, \mathbf{p})$. Let $A\subseteq\mathrm{Gr}$ be the open subset parametrizing complete intersections. Then, $W_1^{(2)}$ is the preimage of $A$ under the composition $W_1\rightarrow\mathcal{V}^n\rightarrow\mathrm{Gr}$, and it is hence open in $W_1$.

Finally, let us prove (b). The restriction $W_1^{(2)} \rightarrow\mathcal{V}^n|_A$ is a dominant morphism whose fibers are isomorphic to $\mathbb{P}^{1}$. So, we can compute that
\[
\dim W_1^{(2)}=\dim\mathbb{P}^1+\dim\mathcal{V}^n|_A=1+n(r-2)+\dim A=1+n(r-2)+2(b(r,d)-2),
\]
which has codimension $n-b(r,d)+2\geq2$ in $\mathcal{U}^n$.
\end{proof}


\subsection{A remark about the Eagon--Northcott bound}

The codimension of determinantal varieties is bounded from above by the Eagon--Northcott bound \cite[Theorem~3]{EN62}. More precisely, if $M$ is an $r\times s$ matrix with entries in a Noetherian ring $R$ and the ideal of $t$-minors of $M$ is proper, then it satisfies $\mathrm{height}(I_t(M))\leq (r-t+1)(s-t+1)$. In our setup, when $R = \Bbbk[x_{ij}]$ and $M = (f_{k\ell})$, this implies that the codimension of the associated determinantal variety of $t$-minors of $M$ is at most $(r-t+1)(s-t+1)$. This upper bound becomes an equality, for instance, for \emph{generic} determinantal varieties, i.e., when $R$ is a polynomial ring in $rs$ variables over a field and the entries of $M$ are distinct variables of $R$. 

Recall the variety $X_{r,d,n}$ is defined by the vanishing of the maximal minors of the matrix $\bM_{r,d,n}(\mathbf{x})$ (see Proposition~\ref{prop:equationsdefiningXrdn}). In our set-up, the codimension of the determinantal variety $X_{r,d,n}$ still satisfies equality in the Eagon--Northcott bound by Corollary~\ref{cor:irranddimofXrdn} (this allowed us to deduce the algebraic properties in Proposition~\ref{prop:XrdnisCMandGorenstein}). On the other hand, the varieties $X_{r,d^m,n}$ for $m>1$, which are defined by the vanishing of the minors of size $b(r,d)-m+1$ of $\bM_{r,d,n}(\mathbf{x})$, do not satisfy the equality in the Eagon--Northcott bound if $n \gg 0$, as it is the case for $X_{2,2^2,n}$ in Example~\ref{ex:reducibleexample}. This is because $X_{r,d^{m},n}$ also parametrizes point configurations that lie on a non-complete intersection of $m$ linearly independent hypersurfaces.

\begin{proposition}\label{prop:smallercodimension}
	If $2\leq m\leq b(r, d-1)$, then $\mathrm{codim}(X_{r,d^m,n})<m(n-b(r,d)+m)$ for $n\gg0$. In particular, $\mathrm{codim}(X_{r,d^m,n})$ does not satisfy the equality in the Eagon--Northcott bound.
\end{proposition}

\begin{proof}
Let $Y_{r,n}\subseteq(\PP^r)^n$ be the subset of $n$-point configurations lying on a hyperplane. Note that this is the closed subvariety cut out by the $(r+1)\times(r+1)$-minors of the $(r+1)\times n$ matrix of coordinates in $(\PP^r)^n$. This is isomorphic to a $(\PP^{r-1})^{n}$-bundle over $\mathrm{Gr}(r, r+1)$, hence it is an irreducible variety of dimension $n(r-1) + r$. 

If $m\leq b(r, d-1)$, then $Y_{r,n}\subseteq X_{r,d^m,n}$. Indeed, let $\bp\in Y_{r,n}$ and let $H\subseteq\PP^d$ be a hyperplane passing through the points in $\bp$. Then, for any $S' \in |\cO_{\PP^{r}}(d-1)|$, $H \cup S'$ is a degree $d$ hypersurface containing $\bp$. Since the map $|\cO_{\PP^{r}}(d-1)| \stackrel{\cup H}{\to} |\cO_{\PP^{r}}(d)|$ is injective, we can find $m \le b(r, d-1)$ linearly independent hypersurfaces containing $\bp$. Hence, $\bp\in X_{r,d^m,n}$. Therefore,
\[
\mathrm{codim}(X_{r,d^m,n})\leq\mathrm{codim}(Y_{r,n})=n-r.
\]
So, for $n \gg 0$, we have that $n - r < m(n-b(r,d)+m)$. Thus, we obtain the desired result.
\end{proof}


\section{Applications and examples}\label{sec:examples}

In this section we apply our results to the case of point configurations on plane curves and quadric hypersurfaces.


\subsection{Point configurations lying on plane curves}
\label{sec:planecase}

As the first nontrivial example, we completely characterize the singular locus of $X_{2,d,n}$.

\begin{theorem}
\label{thm:characterization-sing-points-of-X2dn}
Let $n\geq b(2,d)$ and $\mathrm{char}\;\Bbbk=0$. A point $\bp \in X_{2, d, n}$ is a smooth point if and only if there exists a unique degree $d$ curve $C$ passing through the points in $\bp$ and the points in $\bp$ that lie in the singular locus of $C$ are distinct.
\end{theorem}

\begin{proof}
We know that if there are more than one degree $d$ curve passing through the points in $\mathbf{p}$, or there are coincident points in $\mathbf{p}$ lying in $C^{\sing}$, then $\mathbf{p}\in X_{2,d,n}^{\mathrm{sing}}$ (Theorem~\ref{thm:multiplehypersurfaces} and Theorem~\ref{thm:charactesizationsofsingularandsmoothnpointconfigurations}). So let us consider the case where $C$ is unique and the points in the singular locus are distinct. Let $\bq = \bp \cap C^{\mathrm{sing}}$. Then by \cite[Proposition~2.5]{Laz10} we have that $Z_\mathbf{q}$ is $(d-2)$-normal. Hence, by Corollary~\ref{cor:d-normal-then-d+1-normal}, $Z_\mathbf{q}$ is $d$-normal, which implies $\mathbf{p}$ is a smooth point of $X_{2,d,n}$ by Theorem~\ref{thm:charactesizationsofsingularandsmoothnpointconfigurations}.
\end{proof}

\begin{remark}
With the same notation of the proof of Theorem~\ref{thm:characterization-sing-points-of-X2dn}, if the curve $C$ is irreducible, then a stronger result holds: $Z_\mathbf{q}$ is $(d-3)$-normal, see \cite[Exercise~3.8]{Laz10}.
\end{remark}

\begin{remark}
A simple generalization of Theorem~\ref{thm:characterization-sing-points-of-X2dn} to points on surfaces turns out to be false. Consider $X_{3,2,10}$ parametrizing $10$ points on a quadric surface in $\PP^3$. Consider a surface $S$ given by the union of two distinct planes $H_1,H_2$, which has the line $L=H_1\cap H_2$ as its singular locus. Suppose that $p_1,\ldots,p_4$ are distinct points on $L$, and $p_5,\ldots,p_{10}$ are smooth points of $S$ making $S$ the unique quadric passing through $\bp$. This can be achieved by taking $p_5,p_6,p_7$ general points on $H_1$, and $p_8,p_9,p_{10}$ general points in $H_2$. By Corollary~\ref{cor:singularcriterion2}, $\bp$ is singular. 
\end{remark}


\subsection{The Turnbull--Young hypersurface $X_{3,2,10}$ and beyond}
\label{sec:Turnbull--Young-case}

We now focus our attention to the case of $X_{3,2,10}$. The scheme  $X_{3,2,10}$ is a hypersurface defined by a unique equation $\phi$ in $(\mathbb{P}^3)^{10}$. Being $\mathrm{SL_4}$-invariant, by the Fundamental Theorem of Invariant Theory, the equation $\phi$ can be written as a bracket polynomial, i.e., a polynomial in the maximal minors of the matrix of coordinates of $(\mathbb{P}^3)^{10}$. 
It turns out that this equation, and thus $X_{3,2,10}$, is at the center of the \emph{Turnbull--Young problem}, an old open problem which asks to find a synthetic linear condition (similar to Pascal's theorem) for $10$ points in $\mathbb{P}^3$ to lie on a quadric surface \cite{TY27}. In modern terminology, the Turnbull--Young problem is equivalent to finding a Cayley factorization of $\phi$ in the Grassmann--Cayley algebra. See \cite{Whi90} for a detailed description of the problem and an explicit representation of $\phi$.

Since $\phi$ consists of $240$ bracket monomials, finding the singular locus of $X_{3,2,10}$ by directly applying the Jacobian criterion, i.e., taking the partial derivatives of $\phi$, soon becomes unfeasible. On the other hand, we can take advantage of  Theorem~\ref{thm:charactesizationsofsingularandsmoothnpointconfigurations} to give a geometric characterization of the singular locus of $X_{3,2,10}$. More generally, we can characterize the singular locus of $X_{3,2,n}$ for all $n\geq10$. As we will make use of the classification of quadric surfaces in $\mathbb{P}^3$, for simplicity we will restrict to the case where $\mathrm{char}\; \Bbbk \ne 2$. Note that $X_{3, 2, n} = (\PP^{3})^{n}$ if $n < 10$.

\begin{theorem}\label{thm:sings-of-TY-hypersurface}
Let $n\geq b(3,2)$ and $\mathrm{char}\; \Bbbk \ne 2$. A point $\bp\in X_{3,2,n}$ is smooth if and only if there exists a unique quadric $S\subseteq\PP^3$ containing the points in $\bp$ and one of the following holds:
\begin{enumerate}
\item $S$ is smooth;
\item $S$ is a quadric cone and there is at most one of the points of $\bp$ equal to the singularity of $S$;
\item $S$ is the union of two distinct planes and there are at most three points of $\bp$ in $S^{\sing}$ and they are all distinct.
\end{enumerate}
\end{theorem}

\begin{proof}
By Corollary~\ref{cor:smoothnessXrdn}, it is straightfoward to see that the above conditions imply the smoothness of $X_{3, 2, n}$ at $\bp$.

Conversely, suppose that $\bp$ is a smooth point of $X_{3, 2, n}$. Then there must be a unique quadric $S$ and $Z_{\bq}$  is $2$-normal, where $\bq = \bp \cap S^{\sing}$. If $S$ is smooth or irreducible singular, then we obtain \textit{(1)} or \textit{(2)}. If $S$ is reducible, then to obtain the uniqueness of $S$, it has to be reduced and hence a union of two distinct planes. The singular locus of $S$ is then the line of intersection of the two planes. Since $Z_{\bq}$ is $2$-normal,  by Remark~\ref{rmk:basicpropertiesofregularity}~(4), $Z_{\bq}$ is $2$-normal on the intersection line $S^{\sing}$. A degree two polynomial on a line vanishing at three points must be trivial, so four or more points cannot be $2$-regular. Hence, $|Z_{\bq}| \le 3$. 
\end{proof}

\begin{remark}
We may try to extend the above result to higher dimensional quadrics. Consider $X_{r, 2, n}$ with $r \ge 4$ and $n \ge b(r, 2) = \binom{r+2}{2}$. Note that for any quadric of rank $k$ in $\PP^{r}$, its singular locus is a $(r-k)$-dimensional projective linear space. We expect the regularity of a point configuration $\mathbf{p}\in X_{r,2,n}$ lying on a unique quadric $S$ to be related to the combinatorics of the representable matroid associated to $\mathbf{q}=\mathbf{p}\cap S^{\mathrm{sing}}$. Since the complexity of the matroids grows exponentially, we believe that a complete characterization of the singular locus for all $r$ and $n$ is not a feasible problem. We can, however, give the following partial singularity criterion.
\end{remark}

\begin{proposition}\label{prop:singularonquadric}
Let $\bp \in X_{r, 2, n}$ be a point configuration lying on a unique quadric $S$ and let $\bq := \bp \cap S^{\mathrm{sing}}$. If $\bq$ has a four-secant line, then $\bp$ is a singular point of $X_{r, 2, n}$.
\end{proposition}

\begin{proof}
If $\bq$ has a four-secant line, then $I_{\bq}$ cannot be three-regular by Proposition~\ref{prop:regularityandsecant} and hence $\bp$ is singular by Theorem~\ref{thm:charactesizationsofsingularandsmoothnpointconfigurations}.
\end{proof}


\appendix
\label{sec:appendix}


\section{$d$-normality, $m$-regularity, and secant lines}
\label{sec:d-normality-m-regularity-secant-lines}

Here we collect the definitions of $d$-normality, $m$-regularity and review some of their basic properties that were used in the proofs. Let $\Gamma$  be a finite subscheme of $\mathbb{P}^r$ which is defined by an ideal sheaf $I_{\Gamma}$.
We denote by $|\Gamma|$ the length of $\Gamma$.

\begin{definition}\label{def:independentcondition}
	We say that  $\Gamma$ is \emph{$d$-normal} if  the restriction morphism 
	\[
		\rH^{0}(\cO_{\PP^{r}}(d)) \to \rH^{0}(\cO_{\Gamma}(d))
	\]
	is surjective. Since $\rH^{1}(\cO_{\PP^{r}}(d)) = 0$, this is equivalent to $\rH^{1}(I_{\Gamma}(d)) = 0$. 
\end{definition}

For a finite $\Bbbk$-scheme, we can check its $d$-normality by working over $\overline{\Bbbk}$.

\begin{lemma}\label{lem:basechange}
A finite $\Bbbk$-scheme $\Gamma$ is $d$-normal if and only if its base change $\Gamma_{\overline{\Bbbk}}$ is $d$-normal. 
\end{lemma}

\begin{proof}
Since the field extension map $\mathrm{Spec}\; \overline{\Bbbk} \to \mathrm{Spec}\; \Bbbk$ is flat, by the base change theorem (\cite[Chapter~III, Proposition~9.3]{Har77}), for any coherent sheaf $\mathcal{F}$, 
\[
	\rH^i(\PP^r_{\overline{\Bbbk}},\mathcal{F}_{\overline{\Bbbk}}) \cong \rH^i(\PP^r,\mathcal{F}) \otimes_{\Bbbk} \overline{\Bbbk}.
\]
Thus, $\rH^1(I_{\Gamma}(d)) =0$ if and only if $\rH^1(I_{\Gamma_{\overline{\Bbbk}}}(d)) = 0$. 
\end{proof}

\begin{remark}\label{rmk:basicpropertiesofregularity}
	We list a few basic properties of $d$-normality. The proofs can be found in \cite[Proposition~2.2]{LPW19}.
\begin{enumerate}

\item We have that $\dim \rH^{0}(\cO_{\PP^{r}}(d)) = b(r,d)$ and $\dim\rH^{0}(\cO_{\Gamma}(d))=|\Gamma|$. It follows that $\Gamma$ cannot be $d$-normal if $|\Gamma| > b(r, d)$. On the other hand, if $|\Gamma|\leq b(r,d)$, then $d$-normality is equivalent to $\dim\rH^{0}(I_{\Gamma}(d)) = b(r, d) - |\Gamma|$. In other words, $\Gamma$ is $d$-normal if it imposes $|\Gamma|$ linearly independent conditions on $\rH^{0}(\cO_{\PP^{r}}(d))$.
		
\item If $\Gamma$ is $d$-normal and $\Gamma'$ is a subscheme of $\Gamma$, then $\Gamma'$ is $d$-normal.
		
\item Let $\langle \Gamma \rangle \subseteq \PP^{r}$ be the smallest linear subspace generated by $\Gamma$. Then $\Gamma$ is ($|\Gamma| - \dim \langle \Gamma\rangle)$-normal.
		
\item $\Gamma$ is $d$-normal in $\mathbb{P}^r$ if and only if $\Gamma$ is $d$-normal in $\langle \Gamma \rangle$. To see this, consider the short exact sequence
\[
0 \to I_{\Gamma:H}(-1) \to I_{\Gamma} \to I_{{\Gamma \cap H}/H} \to 0, 
\]
where $H\subseteq\mathbb{P}^r$ is a hyperplane, $I_{\Gamma : H} = (I_{\Gamma}: I_{H})$ is the ideal sheaf of $\Gamma \setminus H$, and $I_{\Gamma \cap H/H}$ is the ideal sheaf of $\Gamma \cap H$ as a subscheme of $H \cong \PP^{r-1}$. From this we obtain that, if $\Gamma \cap H$ is $d$-normal in $H$ and $\Gamma: H$ is $(d-1)$-normal, then $\Gamma$ is $d$-normal. This observation implies the claim after using induction on the number of hyperplanes cutting $\langle\Gamma\rangle$.

\end{enumerate}
\end{remark}

\begin{definition}\label{def:CMregularity}
	Let $\mathcal{F}$ be a coherent sheaf on projective space $\mathbb{P}^r$ and let $m\in\mathbb{Z}_{\geq0}$. We say that $\mathcal{F}$ is \emph{$m$-regular} if $\rH^i(\mathcal{F}(m-i))=0$ for all $i>0$. 
	The \emph{Castelnuovo--Mumford regularity} of $\mathcal{F}$ is the least integer $m$ such that $\mathcal{F}$ is $m$-regular, we denote it by $\reg \mathcal{F}$.
\end{definition}

\begin{remark}
By the same argument of Lemma~\ref{lem:basechange}, we can assume that the base field $\Bbbk$ is algebraically closed when computing $\reg \mathcal{F}$.
\end{remark}

\begin{lemma}\label{lem:d-normality-implies-d+1-regular}
	Let $\Gamma$  be a finite subscheme of $\mathbb{P}^r$ defined by an ideal sheaf $I_{\Gamma}$. Then $\Gamma$ is $d$-normal if and only if $I_{\Gamma}$ is $(d+1)$-regular.
\end{lemma}

\begin{proof}
	If $I_{\Gamma}$ is $(d+1)$-regular, then $\rH^1(I_{\Gamma}(d+1-1))=0$, which means that $\Gamma$ is $d$-normal. Conversely, suppose that $\Gamma$ is $d$-normal. We want to show $\rH^i(I_{\Gamma}(d+1-i))=0$  for all $i>0$. For $i=1$ the vanishing follows immediately by $d$-normality. So assume $i>1$. From the short exact sequence
	\[
	0\rightarrow I_{\Gamma}(d+1-i)\rightarrow\mathcal{O}_{\mathbb{P}^r}(d+1-i)\rightarrow\mathcal{O}_{\Gamma}(d+1-i)\rightarrow 0
	\]
	and since $\rH^i(\mathcal{O}_{\Gamma}(d+1-i))=0$ for $i > 0$ by \cite[Chapter~III, Theorem~2.7]{Har77}, we obtain the isomorphism $\rH^i(I_{\Gamma}(d+1-i))\cong\rH^i(\mathcal{O}_{\mathbb{P}^r}(d+1-i))$, which is zero if $i\neq r$ by \cite[Chapter~III, Theorem~5.1]{Har77}. If $i=r$, then by Serre's duality $\rH^r(\mathcal{O}_{\mathbb{P}^r}(d+1-r))\cong\rH^0(\mathcal{O}_{\mathbb{P}^r}(-d-2))^\vee=0$.
\end{proof}

\begin{corollary}
	\label{cor:d-normal-then-d+1-normal}
Let $\Gamma$  be a finite subscheme of $\mathbb{P}^r$. If $\Gamma$ is $d$-normal, then $\Gamma$ is $(d+1)$-normal.
\end{corollary}

\begin{proof}
	It follows by combining Lemma~\ref{lem:d-normality-implies-d+1-regular} with \cite[Corollary~4.18]{Eis05}.
\end{proof}

The $d$-normality and $(d+1)$-regularity of a finite scheme have a closed connection with the existence of secant lines.

\begin{proposition}\label{prop:regularityandsecant}
Let $\Gamma$ be a reduced finite scheme and suppose that $\Gamma$ has a $m$-secant line $L$, that is, a line passing through $m$ points of $\Gamma$. Then $\mathrm{reg}(I_{\Gamma}) \ge m$. 
\end{proposition}

\begin{proof}
If $I_{\Gamma}$ is $(m-1)$-regular, then $I_{\Gamma}(m-1)$ is globally generated (\cite[Corollary~4.18]{Eis05}). However, if we take a section $s$ of $\rH^{0}(I_{\Gamma}(m-1))$, $V(s) \cap L$ has $m$ points, thus $L \subseteq V(s)$. This implies that $I_{\Gamma}$ cannot be generated by degree $m-1$ polynomials, and hence $\mathrm{reg}(I_{\Gamma}) \ge m$. 
\end{proof}

A refined converse to the above statement is \cite[Theorem~1.1]{LPW19}.



\newcommand{\etalchar}[1]{$^{#1}$}


\begin{thebibliography}{CCRL{\etalchar{+}}20}

\bibitem[ACGH85]{ACGH85}
E.~Arbarello, M.~Cornalba, P.~A. Griffiths, and J.~Harris.
\newblock {\em Geometry of algebraic curves. {V}ol. {I}}, volume 267 of {\em
  Grundlehren der mathematischen Wissenschaften [Fundamental Principles of
  Mathematical Sciences]}.
\newblock Springer-Verlag, New York, 1985.

\bibitem[BCRV22]{BCRV22}
Winfried Bruns, Aldo Conca, Claudiu Raicu, and Matteo Varbaro.
\newblock Determinants, {G}röbner bases, and cohomology, 2022.
\newblock Springer, to appear.

\bibitem[BKSW18]{BKSW18}
Paul Breiding, Sara Kali\v{s}nik, Bernd Sturmfels, and Madeleine Weinstein.
\newblock Learning algebraic varieties from samples.
\newblock {\em Rev. Mat. Complut.}, 31(3):545--593, 2018.

\bibitem[Bri03]{Bri03}
Michel Brion.
\newblock Multiplicity-free subvarieties of flag varieties.
\newblock In {\em Commutative algebra ({G}renoble/{L}yon, 2001)}, volume 331 of
  {\em Contemp. Math.}, pages 13--23. Amer. Math. Soc., Providence, RI, 2003.

\bibitem[BV88]{BV88}
Winfried Bruns and Udo Vetter.
\newblock {\em Determinantal rings}, volume 1327 of {\em Lecture Notes in
  Mathematics}.
\newblock Springer-Verlag, Berlin, 1988.

\bibitem[Car09]{Car09}
Gunnar Carlsson.
\newblock Topology and data.
\newblock {\em Bull. Amer. Math. Soc. (N.S.)}, 46(2):255--308, 2009.

\bibitem[CCRC22]{CCRC22}
Alessio Caminata, Yairon Cid-Ruiz, and Aldo Conca.
\newblock Multidegrees, prime ideals, and non-standard gradings, 2022.
\newblock arXiv:2208.07238.

\bibitem[CCRL{\etalchar{+}}20]{CCRLMZ20}
Federico Castillo, Yairon Cid-Ruiz, Binglin Li, Jonathan Monta\~{n}o, and
  Naizhen Zhang.
\newblock When are multidegrees positive?
\newblock {\em Adv. Math.}, 374:107382, 34, 2020.

\bibitem[CDNG15]{CDNG15}
Aldo Conca, Emanuela De~Negri, and Elisa Gorla.
\newblock Universal {G}r\"{o}bner bases for maximal minors.
\newblock {\em Int. Math. Res. Not. IMRN}, (11):3245--3262, 2015.

\bibitem[CGMS18]{CGMS18}
Alessio Caminata, Noah Giansiracusa, Han-Bom Moon, and Luca Schaffler.
\newblock Equations for point configurations to lie on a rational normal curve.
\newblock {\em Adv. Math.}, 340:653--683, 2018.

\bibitem[CGMS21]{CGMS21}
Alessio Caminata, Noah Giansiracusa, Han-Bom Moon, and Luca Schaffler.
\newblock Point configurations, phylogenetic trees, and dissimilarity vectors.
\newblock {\em Proc. Natl. Acad. Sci. USA}, 118(12):Paper No. 2021244118, 10,
  2021.

\bibitem[CMSV18]{CMSV18}
Aldo Conca, Maral Mostafazadehfard, Anurag~K. Singh, and Matteo Varbaro.
\newblock Hankel determinantal rings have rational singularities.
\newblock {\em Adv. Math.}, 335:111--129, 2018.

\bibitem[CS21]{CS21}
Alessio Caminata and Luca Schaffler.
\newblock A {P}ascal's theorem for rational normal curves.
\newblock {\em Bull. Lond. Math. Soc.}, 53(5):1470--1485, 2021.

\bibitem[CVJ22]{CVJ22}
Gunnar Carlsson and Mikael Vejdemo-Johansson.
\newblock {\em Topological data analysis with applications}.
\newblock Cambridge University Press, Cambridge, 2022.

\bibitem[Deb01]{Deb01}
Olivier Debarre.
\newblock {\em Higher-dimensional algebraic geometry}.
\newblock Universitext. Springer-Verlag, New York, 2001.

\bibitem[Eis95]{Eis95}
David Eisenbud.
\newblock {\em Commutative algebra}, volume 150 of {\em Graduate Texts in
  Mathematics}.
\newblock Springer-Verlag, New York, 1995.
\newblock With a view toward algebraic geometry.

\bibitem[Eis05]{Eis05}
David Eisenbud.
\newblock {\em The geometry of syzygies}, volume 229 of {\em Graduate Texts in
  Mathematics}.
\newblock Springer-Verlag, New York, 2005.
\newblock A second course in commutative algebra and algebraic geometry.

\bibitem[EN62]{EN62}
John~A. Eagon and Douglas~G. Northcott.
\newblock Ideals defined by matrices and a certain complex associated with
  them.
\newblock {\em Proc. Roy. Soc. London Ser. A}, 269:188--204, 1962.

\bibitem[Fef09]{Fef09}
Charles Fefferman.
\newblock Whitney's extension problems and interpolation of data.
\newblock {\em Bull. Amer. Math. Soc. (N.S.)}, 46(2):207--220, 2009.

\bibitem[Ful98]{Ful98}
William Fulton.
\newblock {\em Intersection theory}, volume~2 of {\em Ergebnisse der Mathematik
  und ihrer Grenzgebiete. 3. Folge. A Series of Modern Surveys in Mathematics
  [Results in Mathematics and Related Areas. 3rd Series. A Series of Modern
  Surveys in Mathematics]}.
\newblock Springer-Verlag, Berlin, second edition, 1998.

\bibitem[GJM13]{GJM13}
Noah Giansiracusa, David Jensen, and Han-Bom Moon.
\newblock G{IT} compactifications of {$M_{0,n}$} and flips.
\newblock {\em Adv. Math.}, 248:242--278, 2013.

\bibitem[G\"{a}f20]{Gaf20}
Oliver G\"{a}fvert.
\newblock Computational complexity of learning algebraic varieties.
\newblock {\em Adv. in Appl. Math.}, 121:102100, 26, 2020.

\bibitem[Har77]{Har77}
Robin Hartshorne.
\newblock {\em Algebraic geometry}.
\newblock Graduate Texts in Mathematics, No. 52. Springer-Verlag, New
  York-Heidelberg, 1977.

\bibitem[HE71]{HE71}
Melvin~Hochster and John~A. Eagon.
\newblock Cohen--{M}acaulay rings, invariant theory, and the generic perfection
  of determinantal loci.
\newblock {\em Amer. J. Math.}, 93:1020--1058, 1971.

\bibitem[Hun80]{Hun80}
Thomas~W. Hungerford.
\newblock {\em Algebra}, volume~73 of {\em Graduate Texts in Mathematics}.
\newblock Springer-Verlag, New York-Berlin, 1980.
\newblock Reprint of the 1974 original.

\bibitem[Kwa00]{Kwa00}
Sijong Kwak.
\newblock Generic projections, the equations defining projective varieties and
  {C}astelnuovo regularity.
\newblock {\em Math. Z.}, 234(3):413--434, 2000.

\bibitem[Laz10]{Laz10}
Robert Lazarsfeld.
\newblock A short course on multiplier ideals.
\newblock In {\em Analytic and algebraic geometry}, volume~17 of {\em IAS/Park
  City Math. Ser.}, pages 451--494. Amer. Math. Soc., Providence, RI, 2010.

\bibitem[Liu02]{Liu02}
Qing Liu.
\newblock {\em Algebraic geometry and arithmetic curves}, volume~6 of {\em
  Oxford Graduate Texts in Mathematics}.
\newblock Oxford University Press, Oxford, 2002.
\newblock Translated from the French by Reinie Ern\'{e}, Oxford Science
  Publications.

\bibitem[LPW19]{LPW19}
Wanseok Lee, Euisung Park, and Youngho Woo.
\newblock Regularity and multisecant lines of finite schemes.
\newblock {\em Int. Math. Res. Not. IMRN}, (6):1725--1743, 2019.

\bibitem[Moo15]{Moo15}
Han-Bom Moon.
\newblock Mori's program for {$\overline {\rm M}_{0,6}$} with symmetric
  divisors.
\newblock {\em Math. Nachr.}, 288(7):824--836, 2015.

\bibitem[MR08]{MR08}
Rosa~M. Mir\'{o}-Roig.
\newblock {\em Determinantal ideals}, volume 264 of {\em Progress in
  Mathematics}.
\newblock Birkh\"{a}user Verlag, Basel, 2008.

\bibitem[PSV11]{PSV11}
Daniel Plaumann, Bernd Sturmfels, and Cynthia Vinzant.
\newblock Quartic curves and their bitangents.
\newblock {\em J. Symbolic Comput.}, 46(6):712--733, 2011.

\bibitem[TY27]{TY27}
Herbert~W. Turnbull and Alfred Young.
\newblock The linear invariants of ten quaternary quadrics.
\newblock {\em Trans. Camb. Philos. Soc.}, 23, 1927.

\bibitem[Wat97]{Wat97}
Junzo Watanabe.
\newblock Hankel matrices and {H}ankel ideals.
\newblock {\em Proc. School Sci. Tokai Univ.}, 32:11--21, 1997.

\bibitem[Whi90]{Whi90}
Neil White.
\newblock Implementation of the straightening algorithm of classical invariant
  theory.
\newblock In {\em Invariant theory and tableaux ({M}inneapolis, {MN}, 1988)},
  volume~19 of {\em IMA Vol. Math. Appl.}, pages 36--45. Springer, New York,
  1990.

\end{thebibliography}
\end{document}